\documentclass{amsart}

\usepackage{comment}

\usepackage{amsfonts}
\usepackage{amsthm}
\usepackage{amssymb}
\usepackage{amsmath}
\usepackage{graphicx}
\usepackage{shuffle}
\usepackage{bbm}
\usepackage{multicol}
\usepackage{mathrsfs}
\usepackage{mathtools}
\usepackage{float}
\usepackage{geometry}
\usepackage{hyperref}
\usepackage{tikz}
\usepackage{mathrsfs}
\usepackage{multirow}

\newcommand{\R}{\mathbb R}
\newcommand{\N}{\mathbb N}
\newcommand{\E}{\mathbb E}
\renewcommand{\P}{\mathbb P}

\newcommand{\dif}{\mathrm{d}}

\newcommand{\cov}{\operatorname{cov}}
\renewcommand{\P}{\mathbb{P}}

\allowdisplaybreaks
\usetikzlibrary{matrix,arrows,decorations.pathmorphing}

\newcounter{cprop}[section]

\newtheorem{theorem}[cprop]{Theorem}

\theoremstyle{plain}

\newtheorem{corollary}[cprop]{Corollary}

\newtheorem{lemma}[cprop]{Lemma}
\newtheorem{proposition}[cprop]{Proposition}

\newtheorem{assumption}[cprop]{Assumption}

\numberwithin{equation}{section}

\theoremstyle{definition}
\newtheorem{definition}[cprop]{Definition}
\newtheorem{example}[cprop]{Example}

\theoremstyle{remark}
\newtheorem{remark}[cprop]{Remark}

\newcommand{\vertiii}[1]{{\left\vert\kern-0.25ex\left\vert\kern-0.25ex\left\vert #1 
		\right\vert\kern-0.25ex\right\vert\kern-0.25ex\right\vert}}

\begin{document}
	\title[Introduction to rough paths theory]{Introduction to rough paths theory}
	\author{M. Ghani Varzaneh}
	\address{Mazyar Ghani Varzaneh\\
		Fakult\"at f\"ur  Mathematik und Informatik, FernUniversit\"at in Hagen, Hagen, Germany}
	\email{mazyar.ghanivarzaneh@fernuni-hagen.de}

	\author{S. Riedel}
	\address{Sebastian Riedel \\
		Fakult\"at f\"ur  Mathematik und Informatik, FernUniversit\"at in Hagen, Hagen, Germany}
	\email{sebastian.riedel@fernuni-hagen.de }
	
	%\author{(M. Scheutzow)}
	%\address{(Michael Scheutzow)}

	
	
	\begin{abstract}
		These notes are an extended version of the course ``Introduction to rough paths theory'' given at the XXV Brazilian School of Probability in Campinas in August 2022. Their aim is to give a concise overview to Lyons' theory of rough paths with a special focus on applications to stochastic differential equations. 
	\end{abstract}
	
	\maketitle
	
	\tableofcontents
	
	\section{Introduction}\label{sec_inf_mfd}
	
	Rough paths theory, as we know it today, originates from a series of papers T.~Lyons wrote in the 90s, cf. \cite{Lyo98} and the references therein. In his work, Lyons obtained a deep understanding of paths with low regularity and their interaction within nonlinear systems. One strong motivation for a study of irregular paths is their ubiquity in stochastic analysis. In fact, there are various examples of rescaled random systems that converge to objects with ``rough'' behaviour. The most prominent example is, of course, the \emph{Brownian motion} (Bm) which is a rescaled version of a whole class of random walks. Due to its universality, the Brownian motion plays a key role in stochastic modelling. An important example is a stochastic differential equation in which the ``noise'' is modelled by the (formal) derivative of a Brownian motion. For instance, let us look at the equation
	\begin{align}\label{eqn:stoch_ODE_intro}
		\dif Y_t = \sigma(Y_t) \, \dif B_t(\omega)
	\end{align}
	where $B(\omega)$ is the trajectory of a Brownian motion and $\sigma$ a nonlinear function. Although this equation might look like an innocent non-autonomous random ordinary differential equation, it constitutes a great challenge if we want to analyze it with the tools of classical analysis (we will see in these notes some explanations why this is the case). A major contribution to the understanding of the equation \eqref{eqn:stoch_ODE_intro} was made in the 50s by K.~It\=o who gave a rigorous meaning to it using a \emph{stochastic integral} that nowadays bears his name. It\=o did not view \eqref{eqn:stoch_ODE_intro} as an ordinary differential equation for every trajectory (what is called a \emph{pathwise} point of view) but put forward the \emph{probabilistic} properties of the Brownian motion, namely, its martingale property. His integral is defined using an isometry on a space of martingales, not referring to single trajectories anymore. Eventually, he understood \eqref{eqn:stoch_ODE_intro} as an equation on a space of stochastic processes. It\=o's stochastic calculus was (and is!) extremely successful. Still, a pathwise understanding of \eqref{eqn:stoch_ODE_intro} is desirable in many situations, and one of Lyons' goals was to provide the ground for it. \smallskip
	
	The present text focuses on defining solutions to \emph{rough differential equations} for which \eqref{eqn:stoch_ODE_intro} is a prototype. On the journey to this overall goal, we will touch several key aspects of rough paths theory. These notes are almost self-contained as we give formal proofs of the stated results wherever possible. However, some calculations will be omitted in order not to overload the reader with technical details, but references are given in that case, though. We hope that the reader can use this text to get an idea of what rough paths theory is about, why it was invented and what it can be used for. \smallskip
	
	There are some branches of rough paths theory we were not able to discuss in these notes, and we want to mention two of them here. The first concerns applications of rough paths in the field of stochastic \emph{partial} differential equations. The most famous result here is probably M.~ Hairer's solution to the KPZ-equation that was constructed with the help of rough paths theory \cite{Hai13}. Later, Hairer systematically expanded his ideas and built a whole solution theory for a class of stochastic partial differential equations that he called the theory of \emph{regularity structures} \cite{Hai14}. The reader who is interested in these topics is referred to \cite[Chapter 12 -- 15]{FH20} and \cite{Hai15} for an overview. A second complex we were not able to touch concerns the relationship between rough paths theory and machine learning. In fact, the so-called \emph{signature method} is a very powerful tool that can be used to analyze and forecast very different kinds of data streams. For an introduction to this method, the reader may consult \cite{CK16} and \cite{LM22}. \smallskip
	
	Several monographs about rough paths theory are available now, cf. e.g. \cite{LQ02, LCL07, FV10, FH20}. The structure of our notes has some similarities to \cite{LCL07}, but our notation and the proofs we present are closer to \cite{FH20}. In particular, we wanted to present the important notion of a \emph{controlled path} introduced by Gubinelli \cite{Gub04}, since this concept plays a prominent role also in regularity structures. In this context, we discuss a more recent result about the geometry of controlled paths in Section \ref{sec:controlled_paths_fobs}. In this form, these results did not appear elsewhere yet.

	\subsection{Notation}
	
	A \emph{path} denotes a continuous function defined on a compact interval with values in a topological space. If $E$ is a topological vector space and $X \colon [0,T] \to E$ a path, we call $X_t - X_s$ with $s, t \in [0,T]$ an \emph{increment} of the path. We will use the notation $\delta X_{s,t} \coloneqq X_t - X_s$. If $(E, |\cdot|)$ is a normed space, we define for a function $\Xi$ defined on a simplex $\Xi \colon \{0 \leq s \leq t \leq T \} \to E$ and $\alpha > 0$ the quantity
	\begin{align*}
		\| \Xi \|_{\alpha} \coloneqq \sup_{s < t} \frac{| \Xi_{s,t} |}{|t-s|^{\alpha}}.
	\end{align*}
	If $X \colon [0,T] \to E$ is a path and $\alpha \in (0,1]$, $\| X \|_{\alpha} \coloneqq \| \delta X \|_{\alpha}$ is the usual $\alpha$-H\"older seminorm. A \emph{partition} of an interval $[s,t]$ is a finite set of points $\mathcal{P} = \{s = t_0 < \ldots < t_N = t\}$. We will also view the partition $\mathcal{P}$ as a set of closed intervals $\mathcal{P} = \{[t_i,t_{i+1}] \,:\, i = 0,\ldots,N-1 \}$. The \emph{mesh size} of $\mathcal{P}$ is defined as $|\mathcal{P}| \coloneqq \max_{[u,v] \in \mathcal{P}} |v-u|$. For two Banach spaces $V$ and $W$, $L(V,W)$ denotes the space of continuous linear functions from $V$ to $W$. The space $L(V,W)$ itself is equipped with the operator norm
	\begin{align*}
		\|\Phi\| \coloneqq \sup_{v \neq 0} \frac{| \Phi v |}{|v|}, \quad \Phi \in L(V,W).
	\end{align*}
	
	By $C$, we will mostly mean a generic constant that depends on the 
		aforementioned parameters. If we want to emphasize the dependence on a certain parameter $p$, we use the notation $C_{p}$. In a series of (in-)equalities, the actual value of this constant may change from line to line.

	\section{Motivation: Fractional Brownian motion}
	In this section, we present some background about the \emph{fractional Brownian motion}. These processes form a natural generalization of the \emph{Brownian motion} and were first introduced by Mandelbrot and van Ness in \cite{MvN68}. Let us first recall the definition of a Gaussian process.
	\begin{definition}
		A stochastic process $X \colon [0,\infty) \to \R$ is called \emph{Gaussian} if for every $k \in \N$ and every $t_1,\ldots,t_k \in [0,\infty)$, the random variable $(X_{t_1},\ldots,X_{t_k})$ is a multivariate Gaussian random variable.
	\end{definition}
	Note that the law of a Gaussian process is completely determined by the mean function $\E(X_t)$, $t \in [0,\infty)$, and the covariance function $\cov(X_s,X_t)$, $s,t  \in [0,\infty)$, of the process.
	\begin{definition}[Mandelbrot, van Ness '68]
		Let $H \in (0,1)$. The \emph{fractional Brownian motion} (fBm) is a continuous zero mean Gaussian process $B^H \colon [0,\infty) \to \R$ starting at $0$ with covariance function given by
		\begin{align*}
			R(s,t) \coloneqq \cov(B^H_s,B^H_t) = \E(B^H_s B^H_t) = \frac{1}{2} \left(|t|^{2H} + |s|^{2H} - |t-s|^{2H} \right).
		\end{align*}
		The parameter $H \in (0,1)$ is called \emph{Hurst parameter}.
	\end{definition}
	
	\begin{remark}
		For $H = \frac{1}{2}$, one obtains $R(s,t) = \min\{s,t\}$, i.e. $B^H$ is the usual Brownian motion (Bm). 
	\end{remark}
	
	Below, we show typical trajectories of the fractional Brownian motion with different Hurst parameters.
	
	\begin{figure}[H]
		\includegraphics[width=10cm]{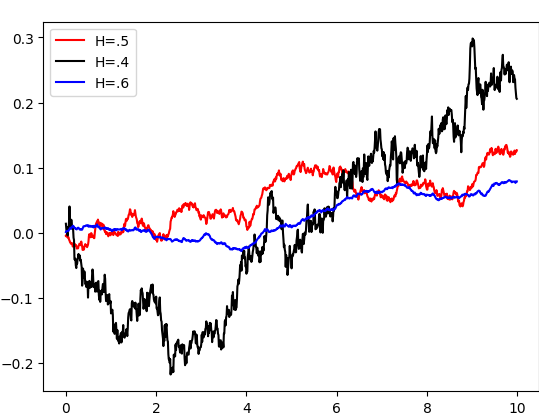}
		\caption{Trajectories of a fractional Brownian motion}
	\end{figure}
	
	For later purposes, we will list some properties of the fractional Brownian motion here. These and others can be found e.g. in \cite[Chapter 5]{Nua06} and \cite{BHOZ08}.
	
	\begin{proposition}
		Let $B^H \colon [0,\infty) \to \R$ be a fractional Brownian motion with Hurst parameter $H \in  (0,1)$. Then the following holds:
		\begin{itemize}
			\item[(i)] $B^H$ has {stationary increments}, i.e. for every $s \geq 0$, we have
			\begin{align*}
				(B^H_{t + s} - B^H_s)_{t \geq 0} \stackrel{\mathcal{D}}{=} (B^H_t)_{t \geq 0}.
			\end{align*}
			\item[(ii)] $B^H$ is self-similar with index $H$, i.e. for every $a > 0$,
			\begin{align*}
				(B^H_{at})_{t \geq 0} \stackrel{\mathcal{D}}{=} (a^{-H} B^H_t)_{t \geq 0}.
			\end{align*}

		\end{itemize}

	\end{proposition}
	
	\begin{proof}
		Exercise.
	\end{proof}
	
	The Hurst parameter describes the behaviour of the process. One easy observation is the following:
	\begin{proposition}
		The increments of the fractional Brownian motion are
		\begin{enumerate}
			\item  uncorrelated for $H = \frac{1}{2}$,
			\item  positively correlated for $H > \frac{1}{2}$,
			\item  negatively correlated for $H < \frac{1}{2}$.
		\end{enumerate}
	\end{proposition}

	%   \begin{center}
		% 	\includegraphics[height=150pt]{fBm.png}
		% \end{center}

	\begin{remark}

		\begin{itemize}
            \item In stochastic modelling, the term \emph{noise} usually denotes the formal derivative of the Brownian motion $B$. To give a rigorous definition, $\dot{B}$ is understood as a random \emph{generalized function} or \emph{distribution}. More precisely,
        \begin{align*}
            \langle \dot{B},\phi\rangle=\int_{0}^{\infty}\phi(s)\mathrm{d}B_{s}
        \end{align*}
        for every smooth function $\phi \colon [0,\infty) \to \R$ with compact support. In particular,
        \begin{align*}
            \E( \langle \dot{B},\phi\rangle  \langle \dot{B},\psi\rangle) = \mathbb{E} \left( \int_{0}^{\infty}\phi(s) \, \mathrm{d}B_{s}\int_{0}^{\infty}\psi(s) \, \mathrm{d}B_{s} \right) = \int_{0}^{\infty}\phi(s)\psi(s)\, \mathrm{d}s.
        \end{align*}
        This suggests that $\E(\dot{B}_t) = 0$ and
        \begin{align}\label{eqn:white_noise}
            \cov(\dot{B}_s, \dot{B}_t) = \E(\dot{B}_s \dot{B}_t) = \delta_{s t} = \begin{cases}
                1   &\text{if } s = t, \\
                0 &\text{otherwise}
            \end{cases} 
        \end{align}
        for every $s,t \in [0,\infty)$.  Note, however, that these identities are only formal since the indicator functions  $\mathbbm{1}_{[0,t]}$ are not smooth and thus not a valid choice for $\phi$ and $\psi$. Still, \eqref{eqn:white_noise} justifies the name \emph{white noise} for $\dot{B}$. For the fractional Brownian motion $B^H$, one can make a similar (formal) calculation that indicates that the process $\dot{B}^H$ is stationary but has non-vanishing correlations for $H \neq \frac{1}{2}$. Sometimes, this kind of noise is called \emph{colored}.
		
			\item  Using the fractional Brownian motion instead of the Brownian motion for modelling random phenomena can be more realistic in case of models with memory. For instance, it was used to model price processes in illiquid markets (electricity markets, gas markets etc.) 
		\end{itemize}
	\end{remark}
	A generic form of a stochastic differential equation (SDE) driven by a fractional Brownian motion is
	\begin{align}\label{eqn:SDE_fBm}
		\begin{split}
			\dif Y_t &= b(Y_t)\, \dif t + \sum_{i = 1}^d \sigma_i(Y_t) \, \dif B^{H;i}_t \\
			%	&\eqqcolon b(Y_t)\, \dif t + \sigma(Y_t) \, \dif B^{H}_t; \quad t \geq 0, \\
			Y_0 &= y_0 \in \R^m
		\end{split}
	\end{align}
	where $B^H = (B^{H;1}, \ldots, B^{H;d})$ is a \emph{$d$-dimensional fractional Brownian motion}, i.e. a vector of independent one-dimensional fractional Brownian motion, $b, \sigma_1,\ldots,\sigma_d \colon \R^m \to \R^m$ is a collection of vector fields and $Y \colon [0,\infty) \to \R^m$ is a stochastic process we aim to call a solution to \eqref{eqn:SDE_fBm}.
	The {fundamental problem} is: \emph{How should we interpret \eqref{eqn:SDE_fBm}?} Or, in other words: \emph{What properties should the process $Y$ satisfy to call it a \emph{solution} to the stochastic differential equation \eqref{eqn:SDE_fBm}}? \bigskip
	
	%\begin{itemize}
	%\item
	
	\textbf{First attempt:}  If the trajectories of $B^H$, i.e. the paths $t \mapsto B^H_t(\omega)$, $\omega \in \Omega$, were differentiable, we could interpret \eqref{eqn:SDE_fBm} \emph{pathwise} as a random (non-autonomous) ordinary differential equation (ODE):
	\begin{align}\label{eqn:pathwise_SDE}
		\frac{\dif Y_t}{\dif t} &= b(Y_t) + \sum_{i = 1}^d \sigma^i(Y_t) \frac{\dif B^{H;i}_t(\omega)}{\dif t}.
	\end{align}
	However, we will see now that this attempt fails.
	
	\begin{lemma}\label{lemma:rogers}
		For a fractional Brownian motion $B^H$ and $p > 0$, we have
		\begin{align}\label{eqn:rogers_sum}
			\sum_{j = 1}^{2^n} |B^H_{j 2^{-n}} - B^H_{(j-1) 2^{-n}}|^p  \stackrel{\P}{\to} \begin{cases}
				0 &\text{if } pH > 1, \\
				\infty &\text{if } pH < 1
			\end{cases}
		\end{align}
		as $n \to \infty$.
		
	\end{lemma}
	
	\begin{proof}
		Define
		\begin{align*}
			Y_n \coloneqq  \sum_{j = 1}^{2^n} |B^H_{j 2^{-n}} - B^H_{(j-1) 2^{-n}}|^p (2^n)^{pH - 1}.
		\end{align*}
		By the scaling property,
		\begin{align*}
			Y_n = \sum_{j = 1}^{2^n} |B^H_{j 2^{-n}} - B^H_{(j-1) 2^{-n}}|^p (2^n)^{pH - 1} \stackrel{\mathcal{D}}{=} \frac{1}{2^n} \sum_{j = 1}^{2^n} |B^H_{j} - B^H_{(j-1)} |^p \eqqcolon \tilde{Y}_n.
		\end{align*}
		Since the fractional Brownian motion has stationary increments, the sequence $(B^H_{j} - B^H_{(j-1)})_{j \geq 1}$ is stationary. Therefore, by Birkhoff's ergodic theorem,
		\begin{align*}
			\tilde{Y}_n \to \E(|B^H_1|^p) \eqqcolon c_p > 0
		\end{align*}
		almost surely and in $L^1$ as $n \to \infty$. It follows that $Y_n \stackrel{\mathcal{D}}{\to} c_p$ and, consequently, $Y_n \stackrel{\P}{\to} c_p$ as $n \to \infty$. From this, the claim follows.

	\end{proof}

	\begin{proposition}
		On any interval $[0,T]$, the fractional Brownian motion is not continuously differentiable almost surely.
	\end{proposition}
	
	\begin{proof}
		By rescaling, we can assume w.l.o.g. that $[0,T] = [0,1]$. Assume that $B^H$ is continuously differentiable with positive probability on $[0,1]$. Then there is a random constant $C > 0$ that is finite with positive probability such that $|B^H_t - B^H_s| \leq C |t-s|$ for every $s,t \in [0,1]$. Therefore,
		\begin{align*}
			\sum_{j = 1}^{2^n} |B^H_{j 2^{-n}} - B^H_{(j-1) 2^{-n}}| \leq \frac{C}{2^n} \sum_{j = 1}^{2^n} j - (j-1) = C < \infty
		\end{align*}
		for every $n  \in \N$ with positive probability which is a contradiction to Lemma \ref{lemma:rogers}.
	\end{proof}
	
	%  \begin{remark}
		%   It is possible to prove the stronger result that with probability one, the fractional Brownian motion is nowhere differentiable on $[0,\infty)$, 
		%  \end{remark}
	
	\begin{remark}
		There is a stronger statement saying that the fractional Brownian motion is nowhere differentiable with probability one that can be deduced from a general result on Gaussian processes, cf. \cite{KK71}. However, we will not need this stronger statement here.
	\end{remark}

	Motivated by partial differential equations, one might have the idea to weaken the notion of differentiability in order to give a meaning to \eqref{eqn:pathwise_SDE}. We could interpret $\frac{\dif B^{H;i}_t(\omega)}{\dif t}$ as a \emph{weak derivative}, i.e. as a \emph{distribution} or \emph{generalized function} \cite{Str03, Eva10}. However, this will lead to another problem: the equation \eqref{eqn:pathwise_SDE} contains products {$\sigma^i(Y_t) \cdot \frac{\dif B^{H;i}_t(\omega)}{\dif t}$} of non-smooth functions with distributions, and such products are (in general) not well defined \cite{Sch54}.
	
	\bigskip

	\textbf{Second attempt:} In stochastic analysis, the It\=o integral $\int Y\, \dif X$ is defined in case of $X$ being a semimartingale and $Y$ being adapted to the filtration generated by $X$. One could try to interpret \eqref{eqn:SDE_fBm} as an integral equation
	\begin{align*}
		Y_t = Y_0 + \int_0^t b(Y_s)\, \dif s + \sum_{i = 1}^d \int_0^t \sigma^i(Y_s)\, \dif B^{H,i}_s 
	\end{align*}
	where the stochastic integral is understood in It\=o-sense. However, one can prove the following:
	\begin{proposition}
		The fractional Brownian motion $B^H$ is not a semimartingale unless $H = \frac{1}{2}$.
	\end{proposition}
	
	\begin{proof}
		This is another consequence of Lemma \ref{lemma:rogers}: If the fractional Brownian motion was a semimartingale, the sum \eqref{eqn:rogers_sum} would converge in probability for $p = 2$ to the quadratic variation process evaluated at $1$. Since this random variable is finite almost surely, this is a contradiction to Lemma \ref{lemma:rogers} in the case $H < \frac{1}{2}$. For $H> \frac{1}{2}$, Lemma \ref{lemma:rogers} implies that the quadratic variation process equals 0 almost surely. This means that the martingale part in the semimartingale decomposition vanishes and that the fractional Brownian motion has almost surely paths of finite variation. This, however, is a contradiction to Lemma \ref{lemma:rogers} when choosing $p = 1$.   
	\end{proof}
	This shows that the classical It\=o approach is not applicable to the fractional Brownian motion, too. \bigskip
	%\item  
	
	\textbf{Third attempt:} In 1936, L.C.~Young introduced a notion of an integral that generalizes Riemann-Stieltjes integration \cite{You36}. More concretely, he defined an integral for functions $f,g \colon [0,T] \to \R$ that are H\"older continuous with H\"older index $\alpha \in (0,1]$ resp. $\beta \in (0,1]$ of the form $\int f\, \dif g$ provided $\alpha + \beta > 1$. To employ this approach, we first need to understand the regularity of the fractional Brownian motion. The following theorem is a classical result:
	\begin{theorem}[Kolmogorov-Chentsov]\label{thm:kolmogorov1}
		Let $X \colon [0,T] \to \R$ be a continuous stochastic process, $q \geq 2$, $\beta > \frac{1}{q}$ and assume that
		\begin{align*}
			\| X_t - X_s \|_{L^q} \leq C|t-s|^{\beta}
		\end{align*}
		for a constant $C > 0$ and any $s,t \in [0,T]$. Then for all $\alpha \in [0, \beta - 1/q)$, there is a random variable $K_{\alpha} \in L^q$ such that
		\begin{align*}
			|X_t - X_s| \leq K_{\alpha} |t-s|^{\alpha}
		\end{align*}
		for all $s,t \in [0,T]$. In particular, the trajectories of $X$ are almost surely $\alpha$-H\"older continuous.
	\end{theorem}
	\begin{proof}
		The proof is classical and can be found e.g. in \cite[(2.1) Theorem]{RY99}. Since we will use similar arguments later for proving Theorem \ref{thm:kolmogorov}, we provide a full proof here.
  
        Without loss of generality, we can assume that $T = 1$. Set
		\begin{align*}
			D_n = \{ k2^{-n}\, :\, k = 0,\ldots, 2^n \} \quad \text{and}  \quad D = \cup_{n \geq 0} D_n.
		\end{align*}
		We further define the random variables
		\begin{align*}
			K_n \coloneqq \sup_{t \in D_n} |\delta X_{t, t+2^{-n}}|, \quad \delta X_{t, t+2^{-n}} = X_{t+2^{-n}} - X_t.
		\end{align*}
		Then it holds that
		\begin{align*}
			\E(K^q_n) \leq \E \sum_{t \in D_n} |\delta X_{t,t+2^{-n}}|^q \leq \frac{1}{|D_n|} C^q |D_n|^{q \beta} = C^q |D_n|^{q\beta - 1}
		\end{align*}
		where $|D_n| = 2^{-n}$. Fix $s < t \in D$ and choose $m$ such that $|D_{m+1}| < t-s \leq |D_m|$. Going from coarser to finer partitions successively, we can find $\tau_0, \ldots,\tau_N \in \cup_{n \geq m+1} D_n$ such that
		\begin{align*}
			s = \tau_0 < \tau_1 < \ldots < \tau_N = t
		\end{align*}
		with the property that at most two intervals of the form $[\tau_i,\tau_{i+1}]$ have the same length.
		With this choice, it follows that
		\begin{align*}
			|\delta X_{s,t}| \leq \sum_{i = 0}^{N-1} |\delta X_{\tau_i, \tau_{i+1}}| \leq 2 \sum_{n \geq m+1} K_n.
		\end{align*}
		We thus obtain
		\begin{align*}
			\frac{|\delta X_{s,t}|}{|t-s|^{\alpha}} \leq \sum_{n \geq m+1} \frac{2 K_n}{|D_{m+1}|^{\alpha}} \leq \sum_{n \geq m+1} \frac{2 K_n}{|D_{n}|^{\alpha}} \leq 2 \sum_{n \geq 0} \frac{K_n}{|D_{n}|^{\alpha}} \eqqcolon K_{\alpha}.
		\end{align*}
		Therefore, we have shown that
		\begin{align*}
			|\delta X_{s,t}| \leq K_{\alpha}|t-s|^{\alpha}
		\end{align*}
		for every $s,t \in D$. By continuity of $X$, this bound holds in fact for every $s,t  \in [0,1]$. It remains to check that $K_{\alpha}$ is in $L^q$. Indeed,
		\begin{align*}
			\| K_{\alpha} \|_{L^q} \leq 2 \sum_{n \geq 0} \frac{\|K_n\|_{L^q}}{|D_n|^{\alpha}} \leq 2C \sum_{n \geq 0} |D_n|^{\beta - \frac{1}{q} - \alpha}
		\end{align*}
		which is summable by assumption on $\alpha$. This proves the theorem.
	\end{proof}

 \begin{remark}
		Often, the formulation of the Kolmogorov-Chentsov theorem does not assume that $X$ is continuous. The statement then says that $X$ has a H\"older-continuous \emph{modification} $\tilde{X}$, i.e. $X_t = \tilde{X}_t$ almost surely for every $t$. Note that the proof above yields the same statement: instead of using continuity of $X$, we define a process $\tilde{X}$ to coincide with $X$ on the dyadic numbers $D$ and extend it continuously to the whole interval $[0,1]$. One can check that $\tilde{X}$ is a modification of $X$.
	\end{remark}
	
 Using the Kolmogorov-Chentsov theorem, we can deduce an important property concerning the trajectories of a fractional Brownian motion:

	\begin{corollary}\label{cor:hoelder_fbm}
		The trajectories of the fractional Brownian motion are almost surely $\alpha$-H\"older continuous for every $\alpha < H$.
	\end{corollary}
	\begin{proof}
		By definition,
			\begin{align*}
				&\quad	\|B^H_t - B^H_s \|_{L^2}^2=\E\big( (B^H_t-B^{H}_s)^2\big)=\E(B^{H}_tB^{H}_t)-2\E(B^{H}_tB^{H}_s)+\E(B^{H}_sB^{H}_s)=|t-s|^{2H}
			\end{align*}
			for every $s,t$. Since $B^H$ is Gaussian, all $L^q$-norms are equivalent. Therefore,
			\begin{align*}
				\|B^H_t - B^H_s \|_{L^q}^{q} \leq C_q^{q}  (\|B^H_t - B^H_s \|_{L^2}^{2})^{\frac{q}{2}}= C_q^{q}|t-s|^{qH},
			\end{align*}
			for every $q \geq 2$. The result now follows from Theorem \ref{thm:kolmogorov1}.
	\end{proof}
	%For later purposes, we give a proof of the Kolmogorov-Chentsov theorem here.

	Corollary \ref{cor:hoelder_fbm} opens the possibility to understand the integral that appears in the integrated equation \eqref{eqn:SDE_fBm} using Young's integration theory. We will follow this approach in the next section.

	\section{Sewing lemma and Young's integral}
	\subsection{The Sewing lemma}
	In rough path theory, the Sewing lemma is one of the cornerstones which will allow us to define integrals. In this part, we present this result and show how it can be used to define Young integrals. Before doing this, we will introduce some more notation.

	\begin{definition}
		Let $W$ be a Banach space. 
		\begin{enumerate}
			\item $\mathcal{C}([0,T],W)$ will denote the space of continuous functions $f \colon [0,T] \to W$.
			\item For $\alpha \in (0,1]$, $\mathcal{C}^{\alpha}([0,T],W)$ is defined as the space of $\alpha$-H\"older continuous functions $f \colon [0,T] \to W$, i.e. $f \in \mathcal{C}^{\alpha}([0,T],W)$ if and only if
			\begin{align*}
				\| f \|_{\alpha} = \sup_{s < t} \frac{| \delta f_{s,t} |}{|t-s|^{\alpha}} < \infty.
			\end{align*} 		
			\item The space $\mathcal{C}_2^{\alpha,\beta}([0,T],W)$ denotes the space of functions $\Xi$ defined on the simplex $\{(s,t) \in [0,T]^2\, :\, s \leq t \}$ such that $\Xi_{t,t} = 0$ and 
			\begin{align*}
				\| \Xi \|_{\alpha, \beta} \coloneqq \| \Xi \|_{\alpha} + \| \delta \Xi \|_{\beta} < \infty
			\end{align*}
			where 
			%		\begin{align*}
				%			\| \Xi \|_{\alpha} \coloneqq \sup_{s<t} \frac{|\Xi_{s,t}|}{|t-s|^{\alpha}}
				%		\end{align*}
			%	and 
			\begin{align*}
				\delta \Xi_{s,u,t} \coloneqq \Xi_{s,t} - \Xi_{s,u} - \Xi_{u,t}, \qquad  \| \delta \Xi \|_{\beta} \coloneqq \sup_{s < u < t} \frac{|\delta \Xi_{s,u,t}|}{|t-s|^{\beta}}.
			\end{align*}
		\end{enumerate}
	\end{definition}
	We can now formulate the Sewing lemma.

	\begin{lemma}[Sewing lemma]\label{lemma:sewing}
		Let $0 < \alpha \leq 1 < \beta$. Then there exists a unique continuous linear map $\mathcal{I} \colon \mathcal{C}^{\alpha,\beta}_2([0,T],W) \to \mathcal{C}^{\alpha}([0,T],W)$ such that $(\mathcal{I} \Xi)_0 = 0$ and
		\begin{align}\label{eqn:sewing}
			|\delta \mathcal{I} \Xi_{s,t} -  \Xi_{s,t}| \leq \| \delta \Xi \|_{\beta}\big{[} 2^{\beta}(\zeta(\beta)-1)+1\big{]} |t-s|^{\beta} 
		\end{align}
		where $C > 0$ depends on $\beta$ and $\zeta$ denotes the Riemann zeta function
		\begin{align*}
			\zeta(s) = \sum_{n=1}^{\infty} \frac{1}{n^s}.
		\end{align*}
		Moreover,
		\begin{align*}
			\delta \mathcal{I} \Xi_{s,t} = \lim_{|\mathcal{P}| \to 0} \sum_{[u,v] \in \mathcal{P}} \Xi_{u,v}.
		\end{align*}

		% and $\| \delta \Xi \|_{\beta}$.
		
	\end{lemma}
	\begin{proof}
		Let us prove uniqueness first. Assume that $I$ and $\tilde{I}$ both satisfy \eqref{eqn:sewing}. Then it holds that
		\begin{align*}
			|(I - \tilde{I})_t - (I - \tilde{I})_s| \leq C |t-s|^{\beta}.
		\end{align*}
		Since $\beta > 1$ and $I - \tilde{I}$ is a path, $I - \tilde{I}$ is constant. Since $I_0 = \tilde{I}_0 = 0$, uniqueness follows.
		Now fix an interval $[s,t]$ and a partition $\mathcal{P} = \{s = u_0 < u_1 < \ldots < u_r =t \}$ of this interval.
		% 	  Let $|\mathcal{P}|$ denote the maximal length between two consecutive points in $\mathcal{P}$, i.e.
		% 	\begin{align*}
			% 		|\mathcal{P}| \coloneqq \max_i |u_{i+1} - u_i|.
			% 	\end{align*}
		We set 
		\begin{align*}
			\int_{\mathcal{P}}\Xi \coloneqq \sum_{[u,v] \in \mathcal{P}} \Xi_{u,v}.
		\end{align*}
		The idea is now to establish a \emph{maximal inequality} for $\int_{\mathcal{P}}\Xi$ by successively removing distinguished points from the partition $\mathcal{P}$. We claim that if $r \geq 3$, there exists a point $u \in \mathcal{P}$ such that for its neighbouring points $u_{-} < u < u_{+} \in \mathcal{P}$,
		\begin{align*}
			|u_{+} - u_{-}| \leq \frac{2}{r-1}|t-s|.
		\end{align*}
		Indeed, otherwise we would have
		\begin{align*}
			2 |t-s| \geq \sum_{u \in \mathcal{P} \setminus\{u_0,u_r\}} |u_{+} - u_{-}| > 2|t-s|
		\end{align*}
		which is a contradiction. Note that for $r=2$, clearly $|u_{+} - u_{-}|\leq |t-s|$.
		With this choice for $\# \mathcal{P} = r+1\geq 4$, we obtain
		\begin{align*}
			\left| \int_{\mathcal{P}} \Xi - \int_{\mathcal{P} \setminus\{u\}} \Xi \right| = | \delta \Xi_{u_{-},u,u_{+}}| \leq \| \delta \Xi \|_{\beta}|u_{+} - u_{-}|^{\beta} \leq  \| \delta \Xi \|_{\beta} \frac{2^{\beta}|t-s|^{\beta}}{(r-1)^{\beta}}.
		\end{align*}
		By successively removing points, we arrive at the uniform bound
		\begin{align}\label{eqn:maximal_inequality}
			\begin{split}
				\sup_{\mathcal{P}} \left| \int_{\mathcal{P}} \Xi -  \Xi_{s,t} \right| &\leq 2^{\beta} |t-s|^{\beta} \| \delta \Xi \|_{\beta} \sum_{k = 2}^{\infty} \frac{1}{k^\beta}+ |t-s|^{\beta}\| \delta \Xi \|_{\beta} \\
				&=   \| \delta \Xi \|_{\beta}\big{[} 2^{\beta}(\zeta(\beta)-1)+1\big{]} |t-s|^{\beta} 
			\end{split}
		\end{align}	
		where the right hand side is finite since $\beta > 1$. We aim to define $\mathcal{I} \Xi$ as the limit $\lim_{|\mathcal{P}| \to 0} \int_{\mathcal{P}} \Xi$ for which we have to prove the existence now. It suffices to show that
		\begin{align*}
			\sup_{\max\{ |\mathcal{P}|, |\mathcal{P}'|\} \leq \varepsilon} \left| \int_{\mathcal{P}} \Xi -  \int_{\mathcal{P}'} \Xi \right| \to 0 \quad \text{as } \varepsilon \to 0.
		\end{align*}
		By adding and subtracting $\int_{\mathcal{P} \cup \mathcal{P}'} \Xi$, we can assume without loss of generality that $\mathcal{P} \subset \mathcal{P}'$. In this case,
		\begin{align*}
			\int_{\mathcal{P}} \Xi -  \int_{\mathcal{P}'} \Xi = \sum_{[u,v] \in \mathcal{P}} \left( \Xi_{u,v} - \int_{\mathcal{P}' \cap [u,v]} \Xi \right).
		\end{align*}
		For $\max\{ |\mathcal{P}|, |\mathcal{P}'|\} = |\mathcal{P}| \leq \varepsilon$, we can use the maximal inequality \eqref{eqn:maximal_inequality} to see that
		\begin{align*}
			\left| \int_{\mathcal{P}} \Xi -  \int_{\mathcal{P}'} \Xi \right| &\leq \| \delta \Xi \|_{\beta}\big{[} 2^{\beta}(\zeta(\beta)-1)+1\big{]} \sum_{[u,v] \in \mathcal{P}} |v - u|^{\beta} = \mathcal{O}(|\mathcal{P}|^{\beta - 1}) = \mathcal{O}(\varepsilon^{\beta - 1}).
		\end{align*}
		This finishes the proof.
	\end{proof}
	%   \begin{remark}
		% 	In the proof of Lemma \ref{lemma:sewing}, we saw that
		% 	\begin{align*}
			% 		\delta \mathcal{I} \Xi_{s,t} = \lim_{|\mathcal{P}| \to 0} \sum_{[u,v] \in \mathcal{P}} \Xi_{u,v}.
			% 	\end{align*}
		% 	
		% \end{remark}
	
	\subsection{Young's integral and differential equations driven by H\"older paths}
	
	We are now ready to state Young's result that generalizes Riemann–Stieltjes integration.
	
	\begin{theorem}[Young integral]\label{theorem:young_integral}
		Let $V$ and $W$ be Banach spaces, $g \in \mathcal{C}^{\alpha}([0,T],V)$ and $f \in \mathcal{C}^{\beta}([0,T],L(V,W))$. Assume that $\alpha + \beta > 1$. Then the integral
		\begin{align*}
			\int_s^t f_u \, \dif g_u \in W
		\end{align*}
		exists as a limit of Riemann sums for every $s < t \in [0,T]$ and we call it the \emph{Young integral}. Moreover, we have the estimate
		\begin{align}
			\left| \int_s^t f_u\, \dif g_u - f_s(g_t - g_s) \right| \leq C\|f\|_{\beta} \|g\|_{\alpha} |t-s|^{\alpha + \beta}
		\end{align}
		where $C > 0$ depends $ \alpha +  \beta$.
	\end{theorem}

	\begin{proof}
		%W.l.o.g we may assume $\|f\|_{\beta} \leq 1$ and $\|g\|_{\alpha} \leq 1$, otherwise we can replace $f$ by $f/ \|f\|_{\beta}$ and $g$ by $g/\|g\|_{\alpha}$. 
		Set 
		\begin{align*}
			\Xi_{s,t} \coloneqq f_s (g_t - g_s).
		\end{align*}
		Then we have
		\begin{align*}
			\delta \Xi_{s,u,t} &= f_s (g_t - g_s) - f_s (g_u - g_s) - f_u (g_t - g_u) \\
			&= -(f_u - f_s)(g_t - g_u),
		\end{align*}
		thus $\| \Xi \|_{\alpha} \leq \|f\|_{\infty}\|g\|_{\alpha} < \infty$ and
		\begin{align*}
			\|\delta \Xi \|_{\alpha + \beta} \leq \|f\|_{\beta} \|g\|_{\alpha} < \infty.
		\end{align*}
		We can therefore apply the Sewing lemma and set
		\begin{align*}
			\int_s^t f_u\, \dif g_u =  \delta \mathcal{I} \Xi_{s,t}.
		\end{align*}
	\end{proof}
	Interpreting the integral in this way, we can give  meaning to differential equations driven by sufficiently regular H\"older paths. Before we formulate the statement, we define a class of functions that will be important for us.
	
	\begin{definition}
		For $k \geq 0$, $\mathcal{C}^k(\R^m, L(\R^d,\R^m))$ denotes the space of bounded $k$-times continuously differentiable functions $\sigma \colon \R^m \to L(\R^d,\R^m)$ with bounded derivatives, i.e. $\sigma = (\sigma_1, \ldots,\sigma_d)$ and every $\sigma_i \colon \R^m \to \R^m$ is bounded and $k$-times continuously differentiable with all derivatives being bounded. For $\sigma \in \mathcal{C}^k = \mathcal{C}^k(\R^m, L(\R^d,\R^m))$, we set 
		\begin{align*}
			\|\sigma\|_{\mathcal{C}^k} \coloneqq \max_{l = 0,\ldots,k} \| D^l \sigma \|_{\infty}.
		\end{align*}
		
	\end{definition}

	% \begin{definition}\label{young_definition}
		% 	Let $X \colon [0,T] \to \R^d$ be a path, $\sigma = (\sigma_1,\ldots,\sigma_d)$ a collection of vector fields $\sigma_i \colon \R^m \to \R^m$ and $y \in \R^m$. We call $Y \colon [0,T] \to \R^m$ a \emph{solution to the Young differential equation}
		% 	\begin{align*}
			% 		\dif Y_t &= \sigma(Y_t)\, \dif X_t; \quad t \in [0,T], \\
			% 		Y_0 &= y,
			% 	\end{align*}
		% 	if $Y$ satisfies the integral equation
		% 	\begin{align*}
			% 		Y_t = y + \sum_{i = 1}^d \int_0^t \sigma_i(Y_s)\, \dif X^i_s 
			% 	\end{align*}
		% 	where the integrals are understood as Young integrals.
		% 
		% \end{definition}
	% The natural question is then to provide conditions under which existence and uniqueness of a solution can be guaranteed. We have the following theorem:

	\begin{theorem}\label{thm:young_ode}
		Let $X \in \mathcal{C}^{\alpha}([0,T],\R^d)$ for some $\alpha > \frac{1}{2}$ and let $\sigma \in \mathcal{C}^2(\R^m, L(\R^d,\R^m))$. Then the integral equation
		\begin{align}\label{eqn:young_ODE}
			Y_t &= y + \int_0^t \sigma(Y_t)\, \dif X_t ;  \quad t \in [0,T]
		\end{align}
		possesses a unique solution  $Y \in \mathcal{C}^{\alpha}([0,T],\R^m)$ for every initial condition $y \in \R^m$. The integral is understood as a Young integral.
	\end{theorem}
	
	\begin{remark}
		The integral equation \eqref{eqn:young_ODE} is often called \emph{Young differential equation} although Young never used his integral to solve differential equations. To the authors knowledge, T.~Lyons was the first who studied differential equations involving the Young integral in \cite{Lyo94}.
	\end{remark}
	
	Before we give the proof of Theorem \ref{thm:young_ode}, we state a technical result that will be needed.

	\begin{lemma}\label{lemma:friz_hairer_75}
		Let $\sigma \in \mathcal{C}^2(\R^m, L(\R^d,\R^m))$ and $T \leq 1$. Then there exists a constant $C_{\alpha,K}$ such that for every $X, Y \in \mathcal{C}^{\alpha}([0,T],\R^d)$ with $\|X\|_{\alpha} \vee \|Y\|_{\alpha} \leq K$,
		\begin{align*}
			\|\sigma(X) - \sigma(Y) \|_{\alpha} \leq C_{\alpha, K} \| \sigma \|_{\mathcal{C}^2} \left(|X_0 - Y_0| + \|X - Y \|_{\alpha} \right).
		\end{align*}
		
	\end{lemma}
	
	\begin{proof}
		It follows by using Taylor's theorem repeatedly, cf. \cite[Lemma 7.5]{FH20} for details.
	\end{proof}

	\begin{proof}[Proof of Theorem \ref{thm:young_ode}]
		The proof is classical and uses a fixed point argument. For $0 < T_0 \leq T$ and $Y \in \mathcal{C}^{\alpha}([0,T],\R^m)$ with $Y_0 = y$, we set 
		\begin{align*}
			\mathcal{M}_{T_0}(Y) \coloneqq \left( t \mapsto y + \int_0^{t} \sigma(Y_s)\, \dif X_s\, ;\, t \in [0,T_0] \right).
		\end{align*}
		{Since $\sigma$ is Lipschitz, the path $t \mapsto \sigma(Y_t)$ is $\alpha$-H\"older continuous and the integral is defined as a Young integral.}
		Thus, $\mathcal{M}_{T_0}$ is in fact a map from $\mathcal{C}_y^{\alpha}([0,T_0],\R^m)$ to itself where $\mathcal{C}_y^{\alpha}([0,T_0],\R^m)$ is the complete metric space of $\alpha$-H\"older paths starting in $y$. We aim to show that it is a contraction. We will not do this on the whole space, but restrict ourselves to the closed unit ball
		\begin{align*}
			\mathcal{B}_{T_0} \coloneqq \{Y \in \mathcal{C}_y^{\alpha}([0,T_0],\R^m)\, :\, \|Y\|_{\alpha} \leq 1 \}.
		\end{align*}
        that is still a closed metric space with the induced metric. We first show that $\mathcal{M}_{T_0}$ leaves $\mathcal{B}_{T_0}$ invariant for $T_0 > 0$ sufficiently small, i.e. $\mathcal{M}_{T_0} \colon \mathcal{B}_{T_0} \to \mathcal{B}_{T_0}$.
		From Theorem \ref{theorem:young_integral},
		\begin{align*}
			\|\mathcal{M}_{T_0}\|_{\alpha} &= \| \int_0^{\cdot} \sigma(Y_s)\, dX_s \|_{\alpha} \\
			&\leq C \|X\|_{\alpha} (T^{\alpha} \|\sigma(Y)\|_{\alpha} + \|\sigma(Y) \|_{\infty}) \\
			&\leq C\|X\|_{\alpha}(T^{\alpha} \|\sigma\|_{\mathcal{C}^1} \|Y\|_{\alpha} + \|\sigma\|_{\mathcal{C}^0}) \\
			&\leq C \|X\|_{\alpha}
		\end{align*}
		where $\|X\|_{\alpha}$ denotes the $\alpha$-H\"older norm on $[0,T_0]$. We aim to choose $T_0$ sufficiently small such that $C \| X \|_{\alpha} \leq 1$. However, it is in general \emph{not} true that $\| X \|_{\alpha}$ gets small as $T_0$ tends to $0$ (to see this, take the $\frac{1}{2}$-H\"older norm for the square root function, for instance). To solve this issue, we choose $\alpha'$ such that $\frac{1}{2} < \alpha' < \alpha$ and repeat the calculation for $\alpha'$. If $X$ is $\alpha$-H\"older, it follows that $\| X\|_{\alpha'} \to 0$ as $T_0 \to 0$, thus we can choose $T_0$ small enough to conclude that $\|\mathcal{M}_{T_0}\|_{\alpha'} \leq 1$ and therefore $\mathcal{M}_{T_0} \colon \mathcal{B}_{T_0} \to \mathcal{B}_{T_0}$. We proceed showing that $\mathcal{M}_{T_0}$ is a contraction. For $Y,\tilde{Y} \in  \mathcal{B}_{T_0}$, Theorem \ref{theorem:young_integral} implies that 
		\begin{align*}
			&\left| \delta \mathcal{M}_{T_0}(Y)_{s,t} - \delta \mathcal{M}_{T_0}(\tilde{Y})_{s,t} \right| \\
			=\ &\left| \int_s^t \sigma(Y_u) - \sigma(\tilde{Y}_u) \, \dif X_u \right| \\
			\leq\ &C \left(\|\sigma(Y) - \sigma(\tilde{Y}) \|_{\infty} + \|\sigma(Y) - \sigma(\tilde{Y}) \|_{\alpha'}\right) \| X \|_{\alpha' ; [0,T_0]}|t-s|^{\alpha'}. \\
			%   \leq\ &C \|\sigma\| \|Y - \tilde{Y} \|_{\alpha} \| X \|_{\alpha ; [0,T_0]} |t-s|^{\alpha},
		\end{align*}
		Note that, since $Y_0 = \tilde{Y}_0$,
		\begin{align*}
			\|\sigma(Y) - \sigma(\tilde{Y}) \|_{\infty} &\leq |\sigma(Y_0) - \sigma(\tilde{Y}_0)| + T_0^{\alpha'} \|\sigma(Y) - \sigma(\tilde{Y}) \|_{\alpha'}\\
			&=  T_0^{\alpha'} \|\sigma(Y) - \sigma(\tilde{Y}) \|_{\alpha'}.
		\end{align*}
		From Lemma \ref{lemma:friz_hairer_75},
		\begin{align*}
			\|\sigma(Y) - \sigma(\tilde{Y}) \|_{\alpha'} &\leq C_{\alpha',K} \|\sigma\|_{\mathcal{C}^2} (|Y_0 - \tilde{Y}_0| + \|Y - \tilde{Y}\|_{\alpha'}) \\
			&= C_{\alpha',K} \|\sigma\|_{\mathcal{C}^2}  \|Y - \tilde{Y}\|_{\alpha'}
		\end{align*}
		where $K > 0$ satisfies $\|Y\|_{\alpha} \vee \|\tilde{Y}\|_{\alpha} \leq K$. Since $Y,\tilde{Y} \in \mathcal{B}_{T_0}$, $C$ can be chosen independently of $Y$ and $\tilde{Y}$. Therefore, we arrive at an estimate of the form
		\begin{align*}
			\|\mathcal{M}_{T_0}(Y) - \mathcal{M}_{T_0}(\tilde{Y}) \|_{\alpha'} \leq C \| X \|_{\alpha'} \|Y - \tilde{Y} \|_{\alpha'}
		\end{align*}
		and choosing $T_0 >0$ smaller if necessary, we obtain $C \| X \|_{\alpha'} < 1$, i.e. $\mathcal{M}_{T_0}$ is a contraction on the space $\mathcal{B}_{T_0}$.  It follows that the equation possesses a unique solution on the interval $[0,T_0]$. We can now repeat the argument on the interval $[T_0, 2T_0]$ with initial condition $Y_{T_0}$ and glue together the solutions. Iterating this sufficiently often, we eventually obtain a unique solution $Y$ on the interval $[0,T]$. A posteriori, the estimates for the Young integral in Theorem \ref{theorem:young_integral} show that $Y$ is not only $\alpha'$ H\"older, but even $\alpha$-H\"older continuous. This finishes the proof.
	\end{proof}
	Combining Corollary \ref{cor:hoelder_fbm} and Theorem \ref{thm:young_ode}, we can show:

	\begin{theorem}
		Let $B^H$ be a fractional Brownian motion with $H > \frac{1}{2}$. Assume that $\sigma \in \mathcal{C}^2(\R^m, L(\R^d,\R^m))$. Then for every $y \in \R^m$, the stochastic differential equation
		\begin{align}
			\begin{split}\label{eqn:young_SDE}
				\dif Y_t &= \sigma(Y_t)\, \dif B^H_t(\omega); \quad t \in [0,T], \\
				Y_0 &= y,
			\end{split}
		\end{align}
		can be interpreted as an integral equation using the Young integral and possesses a unique solution $Y$ for almost every trajectory.
		
	\end{theorem}

	\begin{remark}
		%Often, in the definition of the fBm (or the Bm), one imposes the weaker condition that the trajectories of the process are continuous only outside of a set of measure zero. 
		The solution theory just presented is \emph{pathwise}, meaning that one can solve the stochastic differential equation \eqref{eqn:young_SDE} path-by-path. In particular, if the fractional Brownian motion is $\alpha$-H\"older continuous outside a set $\mathcal{N} \subset \Omega$ of measure zero, the equation \eqref{eqn:young_SDE} can be solved outside exactly that set. This is in contrast to It\=o's theory of stochastic differential equations which is \emph{not} pathwise. The solution of an It\=o stochastic differential equation is defined outside a set of measure zero \emph{that depends on the whole equation}, e.g. on $y$ and on $\sigma$, too. Considering another initial condition $\tilde{y}$ will create a new set of measure zero outside of which the solution is defined. Since the set of allowed initial conditions is not countable, it is a priori not clear whether there exists a set of full measure on which an It\=o stochastic differential equation can be solved for every initial condition $y$. In fact, assuming that $\sigma$ is globally Lipschitz continuous, such a universal set always exists, but there are examples of solutions to It\=o stochastic differential equations that fail to have this property, cf. \cite{LS11}. For a pathwise solution theory, this  cannot happen.
	\end{remark}
	
	\subsection{Limitations of the Young integral}
	
	While the Young integral can be used successfully for the fractional Brownian motion in the case $H>\frac{1}{2}$, it cannot be applied even to the {Brownian motion}. The main obstacle is the lack of sufficient regularity, which is essential for defining the integral. We indeed have the following statement for the fractional Brownian motion:
	
	\begin{proposition}
		The fractional Brownian motion $B^H$ does not have $\alpha$-H\"older continuous trajectories on $[0,T]$ almost surely for $\alpha > H$. 
	\end{proposition}
	\begin{proof}
		Can be deduced from Lemma \ref{lemma:rogers}. The details are left to the reader.
	\end{proof}
	
	It is natural to look for an extension of the Young integral that works for paths with low H\"older-regularity, too. More generally, the minimal condition on the notion of an integral would be that we can apply it to the Brownian motion. We make the following (very general) definition:

	\begin{definition}\label{Wiener measure}
		Let $E$ be a Banach space of paths in $\R$ and $(C_n)_{n \geq 1}$ and $(S_n)_{n \geq 1}$ be two series of independent standard Gaussian random variables. Let $c_n(t) \coloneqq \cos(2\pi n t)$ and $s_n(t) \coloneqq \sin(2\pi nt)$. We say that $E$ \emph{carries the Wiener measure} if and only if $s_n$ and $c_n$ belong to $E$ and if the series
		\begin{align*}
			\sum_{n = 1}^{\infty} \frac{C_n c_n + S_n s_n}{2 \pi n}
		\end{align*}
		converges in $E$ almost surely.

	\end{definition}
	%	\begin{comment}

			\begin{example}\label{example:wiener_measusure}
				Let us motivate Definition \ref{Wiener measure} by showing that $L^2[0,1]$ carries the Wiener measure. Assume that $B = (B_{t})_{0\leq t\leq 1}$ is a Brownian motion, i.e. $\E(B_t) = 0$ and $\E(B_s B_t) = \min\{s,t\}$ for every $s,t \in [0,1]$. Recall that $\lbrace 1,\sqrt{2}c_{n},\sqrt{2}s_{n}\rbrace_{n\geq 1}$ is an orthonormal basis for $L^{2}[0,1]$. By $\langle \cdot , \cdot \rangle$, we denote the inner product of two functions in $L^{2}[0,1]$. If we expand $B$ with respect to this basis, we obtain
				\begin{align*}
					B = \int_{0}^{1} B_{s} \, \mathrm{d}s + \sum_{n=1}^{\infty}\left[\langle B,\sqrt{2}c_{n}\rangle \sqrt{2}c_{n} + \langle B,\sqrt{2}s_{n}\rangle \sqrt{2}s_{n} \right]
				\end{align*}
				almost surely in $L^2([0,1])$. Since $B$ is a zero mean Gaussian process, it is easy to check that $\big\lbrace \langle B, 1\rangle,\langle B, \sqrt{2}c_{n}\rangle,\langle B, \sqrt{2}s_{n}\rangle\big\rbrace_{n\geq 1}$ is a family of normal random variables with zero mean. Therefore, the law of each random variable is determined by its second moment. We calculate
				\begin{align*}
					\mathbb{E}\big(\langle B, \sqrt{2}c_n\rangle^2 \big) &= \mathbb{E} \left( \int_{0}^{1}\int_{0}^{1}{2 B_{t} B_{s}}\cos(2\pi nt)\cos(2\pi ns) \, \mathrm{d}t \, \mathrm{d}s \right)\\
                    &=\int_{0}^{1}\int_{0}^{1} 2\min\lbrace s,t\rbrace\cos(2\pi nt)\cos(2\pi ns) \, \mathrm{d}t \, \mathrm{d}s \\
                    &=\int_{0}^{1}\int_{0}^{s}2t\cos(2\pi nt)\cos(2\pi ns) \, \mathrm{d}t \, \mathrm{d}s + \int_{0}^{1}\int_{s}^{1}2s\cos(2\pi nt)\cos(2\pi ns) \, \mathrm{d}t \, \mathrm{d}s \\
                    &=\frac{1}{(2\pi n)^2}.
				\end{align*}
				Similarly, $\mathbb{E}(\langle B, \sqrt{2}s_n\rangle^2 )=\frac{1}{(2\pi n)^2}.$ Proceeding with similar calculations, we see that
				\begin{align*}
					\mathbb{E}\big(\langle B, 1\rangle\langle B, \sqrt{2}c_n\rangle\big) &= \mathbb{E}\big(\langle B, 1\rangle\langle B, \sqrt{2}s_n\rangle\big) = 0 \quad \text{for every } n \geq  1,\\
					\mathbb{E}\big(\langle B, \sqrt{2}c_n\rangle \langle B, \sqrt{2}s_m\rangle\big) &= 0 \quad \text{for every } n, m\geq 1 \text{ and}\\
					\mathbb{E}\big(\langle B, \sqrt{2}c_n\rangle \langle B, \sqrt{2}c_m\rangle\big) &= \mathbb{E}\big(\langle B, \sqrt{2}s_n\rangle \langle B, \sqrt{2}s_m\rangle\big) = 0 \quad \text{for every } n \neq m.
				\end{align*}
				These calculations reveal that the elements $\big\lbrace \langle B, 1\rangle,\langle B, \sqrt{2}c_{n}\rangle,\langle B, \sqrt{2}s_{n}\rangle\big\rbrace_{n\geq 1}$ are uncorrelated and therefore, since they are normal, independent. Setting 
				\begin{align*}
					C_{n} \coloneqq 2\pi n \langle B, \sqrt{2}c_n\rangle  \quad \text{and} \quad S_{n} \coloneqq 2\pi n \langle B, \sqrt{2}s_n\rangle,
				\end{align*}
				we can thus represent $B$ by
				\begin{align*}
					B = \int_{0}^{1} B_{s} \, \mathrm{d}s + \sum_{n = 1}^{\infty} \frac{C_n c_n + S_n s_n}{\sqrt{2} \pi n}
				\end{align*}
				almost surely in $L^2[0,1]$. In fact, with more work, one can even show that the convergence of the series holds uniformly almost surely. This shows that $L^2[0,1]$ and $\mathcal{C}[0,1]$ carry the Wiener measure. One can also show that $\mathcal{C}^{\alpha}[0,1]$ carries the Wiener measure for $\alpha < \frac{1}{2}$ but not for $\alpha \geq \frac{1}{2}$.
			\end{example}
		%	\end{comment}
	\begin{comment}
		\begin{example}\label{example:karhunen_loeve}
			\begin{enumerate}
				\item The series
				\begin{align*}
					\sum_{n = 1}^{\infty} S_n \frac{\sqrt{2} \sin( \pi n t)}{ \pi n}
				\end{align*}
				converges almost surely uniformly in $t \in [0,1]$ to a \emph{Brownian bridge} $B$, i.e. to $B_t = W_t - t W_1$ where $W$ denotes a Brownian motion. This expansion is called \emph{Karhunen-Lo\`eve expansion} of the Brownian bridge.
				
				\item The series
				\begin{align*}
					\sum_{n = 1}^{\infty} C_n \frac{\sqrt{2} \cos( \pi n t)}{ \pi n}
				\end{align*}
				converges almost surely uniformly in $t \in [0,1]$ to a \emph{mean-centered Brownian motion} $W^0_t = W_t - \int_0^1 W_u \, \dif u$.
			\end{enumerate}
			For a proof of these statements, cf. \cite{Deh06} and the references therein.
		\end{example}
		\begin{remark}
			From Example \ref{example:karhunen_loeve}, we can deduce that $\mathcal{C}([0,1],\R)$ carries the Wiener measure. One can also show that $\mathcal{C}^{\alpha}([0,1],\R)$ carries the Wiener measure for $\alpha < \frac{1}{2}$ but not for $\alpha \geq \frac{1}{2}$.
		\end{remark}
	\end{comment}
	The following result is taken from \cite{LCL07}.
	\begin{theorem}[Lyons]\label{Lyons}
		Let $E$ be a Banach space that carries the Wiener measure. Then there is no continuous bilinear map $I \colon E \times E \to \R$ such that if $x$ and $y$ are trigonometric functions, $I(x,y) = \int_0^1 x_t \, \dif y_t$.
	\end{theorem}
	\begin{proof}
		For $N \geq 1$, we define
		\begin{align*}
			W^N \coloneqq \sum_{n = 1}^{N} \frac{C_n c_n + S_n s_n}{2 \pi n} \quad \text{and} \quad  \tilde{W}^N \coloneqq \sum_{n = 1}^{N} \frac{C_n s_n - S_n c_n}{2 \pi n}
		\end{align*}
		where in both definitions, we take the same random variables $S_n$ and $C_n$. Note that $W^N \stackrel{\mathcal{D}}{=} \tilde{W}^N$ and that both $W^N$ and $\tilde{W}^N$ converge almost surely to processes $W$ resp. $\tilde{W}$ in $E$ by assumption.  We assume that a bilinear map $I \colon E \times E \to \R$ satisfying the stated conditions exists. Then we have $I(W^N,\tilde{W}^N) \to I(W,\tilde{W})$ almost surely as $N \to \infty$. On the other hand,
		\begin{align*}
			I(W^N,\tilde{W}^N) = \int_0^1 W^N_t\, \dif \tilde{W}^N_t = \sum_{n = 1}^N \frac{C_n^2 + S_n^2}{2\pi n}
		\end{align*}
		diverges almost surely as $N \to \infty$ which is a contradiction.
	\end{proof}
	\begin{remark}
		In view of Example \ref{example:wiener_measusure}, we see that the processes $W$ and $\tilde{W}$ are both (essentially) the sum of a Brownian motion and a smooth random function. Note that the nonexistence of the integral $\int_0^1 W_t \, \dif \tilde{W}_t$ does not contradict It\=o's theory of stochastic integration: the processes $W$ and $\tilde{W}$ are highly correlated and $W$ is not adapted to the filtration generated by $\tilde{W}$.
	\end{remark}

	Theorem \ref{Lyons} shows that we cannot expect to find a linear theory of deterministic integration that is rich enough to handle Brownian sample paths. The fact that a pathwise approach to stochastic differential equations driven by a Brownian motion seems impossible underlines the importance of It\=o's theory of stochastic integration and is one of many reasons for its tremendous success.

	\section{Rough paths and linear equations}
	Our goal is to solve differential equations driven by paths with regularity less than Brownian sample paths. We already saw that a direct approach using the Young integral will not work. To simplify the problem, we will consider linear equations first. For a $d$-dimensional path $X = (X^1,\ldots,X^d)$, we look at the equation
	\begin{align}\label{eqn:linear_RDE}
		\begin{split}
			\dif Y_t  &= Y_t\, \dif X_t = \sum_{i = 1}^d Y_t \, \dif X^i_t ; \qquad t \geq 0,\\
			Y_0 &= y \in \R.
		\end{split}
	\end{align}
	\emph{Formally}, a solution to \eqref{eqn:linear_RDE} is given by
	\begin{align*}
		Y_t = y &+ y \sum_{i = 1}^d \int_0^t \dif X^i_s + y \sum_{i,j = 1}^d \int_0^t \int_0^s \dif X^i_u \dif X^j_s \\
		&+ \ldots + y \sum_{i_1,\ldots,i_n = 1}^d \int_{0 < t_1 < \cdots < t_n < t} \dif X_{t_1}^{i_1} \cdots \dif X_{t_n}^{i_n} + \ldots
	\end{align*}
	The problem is, of course, that there is no good notion of an integral we can use to define the iterated integrals for irregular paths $X$. On the other hand, there are situations where iterated integrals are given in a non-pathwise manner. For instance, in stochastic analysis, the It\=o and the Stratonovich integral are defined for a Brownian motion. The idea of rough paths theory is to just \emph{assume} that the iterated integrals exist and satisfy some key properties. In this chapter, we will discuss these properties and see how linear equations can be solved. Eventually, we will consider the case of a Brownian motion.
	
	\subsection{Iterated integrals and rough paths}
	What are the properties that characterize an iterated integral? To answer this question, let us start with the second iterated integral. For smooth $X$, we use the notation
	\begin{align*}
		\mathbb{X}^{ij}_{s,t} \coloneqq \int_s^t \int_s^u \dif X^i_v\, \dif X^j_u.
	\end{align*}
	One basic algebraic property is additivity of the integral, i.e. $\int_s^u + \int_u^t = \int_s^t$ for $s < u < t$. For the iterated integral, this leads to 
	\begin{align*}
		\mathbb{X}^{ij}_{s,t} &= \int_s^t(X^i_v - X^i_s)\, \dif X^j_v = \int_s^u(X^i_v - X^i_s)\, \dif X^j_v + \int_u^t(X^i_v - X^i_s)\, \dif X^j_v \\
		&= \int_s^u(X^i_v - X^i_s)\, \dif X^j_v +  \int_u^t(X^i_v - X^i_u)\, \dif X^j_v + \int_u^t(X^i_u - X^i_s)\, \dif X^j_v \\
		&= \mathbb{X}^{ij}_{s,u} + \mathbb{X}^{ij}_{u,t} + \int_s^u \, \dif X^i_v \int_u^t \, \dif X^j_v.
	\end{align*}
	Setting $\mathbb{X}_{s,t} \coloneqq (\mathbb{X}^{ij}_{s,t})_{i,j = 1,\ldots,d} \in \R^{d \times d} \cong \R^d \otimes \R^d$, the above equality reads
	\begin{align*}
		\mathbb{X}_{s,t} = \mathbb{X}_{s,u} + \mathbb{X}_{u,t} + \delta X_{s,u} \otimes \delta X_{u,t}.
	\end{align*}
	To describe the corresponding property for the higher-order iterated integrals, we will introduce some more notations. Note that the $n$-th order iterated integral of a $d$-dimensional smooth path can be understood as an element in $(\R^d)^{\otimes n}$. Thus the collection of all iterated integrals will be an element in the direct product of all tensor products.
	\begin{definition}
		The direct product
		\begin{align*}
			T((\R^d)) \coloneqq \R \times \R^d \times (\R^d \otimes \R^d) \times \cdots \times (\R^d)^{\otimes n} \times \cdots = \prod_{n = 0}^{\infty} (\R^d)^{\otimes n}
		\end{align*}
		where $(\R^d)^{\otimes 0} = \R$, $(\R^d)^{\otimes 1} = \R^d$, is called \emph{extended tensor algebra}. The maps
		\begin{align*}
			\pi_n  \colon T((\R^d)) \to (\R^d)^{\otimes n}
		\end{align*}
		are the usual projection maps.
	\end{definition}
	\begin{definition}\label{tensor}
		For elements $a, b \in T((\R^d))$, we define an element $a \otimes b \in T((\R^d))$ by setting
		\begin{align*}
			\pi_n(a \otimes b) \coloneqq \sum_{i + j = n} \pi_i(a) \otimes \pi_j(b)
		\end{align*}
		for every $n \in \N_0$. We also define
		\begin{align*}
			1 \coloneqq (1,0,0, \ldots) \in T((\R^d)).
		\end{align*}
		
	\end{definition}
	Note that the extended tensor algebra carries a natural vector space structure. It becomes a \emph{associative unital algebra} with product $\otimes$ and $1$ as the unit.

	\begin{definition}
		Let $X \colon [0,T] \to \R^d$ be a smooth path. We define the \emph{(canonical) lift} $\mathbf{X}$ of $X$ as a map $\mathbf{X} \colon \Delta \to T((\R^d))$, $\Delta \coloneqq \{(s,t) \in [0,T]^2\, :\, s \leq t \}$, by setting
		\begin{align*}
			\pi_n(\mathbf{X}_{s,t}) \coloneqq \mathbb{X}_{s,t}^{(n)} &\coloneqq \int_{s < u_1 < \cdots < u_n < t} \dif X_{u_1} \otimes \cdots \otimes \dif X_{u_n} \\
			&\coloneqq \left(\int_{s < u_1 < \cdots < u_n < t} \dif X^{i_1}_{u_1}  \cdots  \dif X^{i_n}_{u_n} \right)_{i_1,\ldots,i_n \in \{1,\ldots,d\}} \in (\R^d)^{\otimes n}
		\end{align*}
		for $n \geq 1$ and $\pi_0(\mathbb{X}_{s,t}) \coloneqq 1$.
		
	\end{definition}
	We can now prove an algebraic property called \emph{Chen's identity} that is satisfied by iterated integrals.

	\begin{theorem}[Chen]
		Let $X \colon [0,T] \to \R^d$ be smooth and $\mathbf{X}$ its canonical lift. Then
		\begin{align*}
			\mathbf{X}_{s,t} = \mathbf{X}_{s,u} \otimes \mathbf{X}_{u,t} 
		\end{align*}
		for every $s < u < t$.
		
	\end{theorem}
	
	\begin{proof}
		We have to show that for every $n \in \N$,
		\begin{align*}
			\mathbb{X}_{s,t}^{(n)} = \sum_{i + j = n} \mathbb{X}_{s,u}^{(i)} \otimes \mathbb{X}_{u,t}^{(j)}.
		\end{align*}
		We do this by induction. For $n=1$, the statement is obvious. For arbitrary $n \geq 2$,
		\begin{align*}
			\mathbb{X}_{s,t}^{(n)} &= \int_s^t \mathbb{X}_{s,v}^{(n-1)} \, \otimes \dif X_v \\
			&= \int_s^u \mathbb{X}_{s,v}^{(n-1)} \, \otimes \dif X_v + \int_u^t \mathbb{X}_{s,v}^{(n-1)} \, \otimes \dif X_v \\
			&= \mathbb{X}_{s,u}^{(n)} + \sum_{i + j = n-1} \int_u^t \mathbb{X}_{s,u}^{(i)} \otimes \mathbb{X}_{u,v}^{(j)} \, \otimes \dif X_v \\
			&= \mathbb{X}_{s,u}^{(n)} + \sum_{i + j = n-1} \mathbb{X}_{s,u}^{(i)} \otimes \int_u^t  \mathbb{X}_{u,v}^{(j)} \, \otimes \dif X_v \\
			&= \mathbb{X}_{s,u}^{(n)} + \sum_{i + j = n-1} \mathbb{X}_{s,u}^{(i)} \otimes \mathbb{X}_{u,t}^{(j + 1)} \\
			&= \sum_{i + j = n} \mathbb{X}_{s,u}^{(i)} \otimes \mathbb{X}_{u,t}^{(j)}
		\end{align*}
		and the claim is shown.

	\end{proof}
	From Chen's theorem, we can deduce that
	
	\begin{align*}
		\mathbb{X}^{(2)}_{s,t} = \pi_2(\mathbf{X}_{s,u} \otimes \mathbf{X}_{u,t}) =  \mathbb{X}^{(2)}_{s,u} + \mathbb{X}^{(2)}_{u,t} + \delta X_{s,u} \otimes \delta X_{u,t}
	\end{align*}
	which we calculated ``by hand'' above.
	\begin{definition}
		For $N \geq 0$, the direct sum
		\begin{align*}
			T^N(\R^d) \coloneqq \R \oplus \R^d \oplus (\R^d \otimes \R^d) \oplus \cdots \oplus (\R^d)^{\otimes N} = \bigoplus_{n = 0}^{N} (\R^d)^{\otimes n}
		\end{align*}
		if called \emph{truncated tensor algebra} of level $N$. The truncated tensor product $\otimes^{N}:T^N(\R^d)\times T^N(\R^d)\rightarrow T^N(\R^d)$
			is then defined similarly as in Definition \ref{tensor} by truncating each term of the product to level $N$. Abusing notation, we will still use the symbol $\otimes$ instead of $\otimes^N$ on the truncated tensor algebra.
	\end{definition}
	The truncated tensor algebra is also an associative unital algebra with sum and product induced by the extended tensor algebra.
	
	\begin{definition}
		A map $\mathbf{X} \colon \Delta \to T^N(\R^d)$ satisfying the \emph{Chen relation}
		\begin{align*}
			\mathbf{X}_{s,t} = \mathbf{X}_{s,u} \otimes \mathbf{X}_{u,t} 
		\end{align*}
		for every $s < u < t$ is called a \emph{multiplicative functional}.
	\end{definition}
	Multiplicative functionals satisfy an \emph{algebraic} property that we expect from iterated integrals. There is also an \emph{analytic} property an iterated integral should satisfy. Recall that for the Young integral, we showed that for an $\alpha$-H\"older path $X$ with $\alpha > \frac{1}{2}$,
	
	\begin{align*}
		|\delta X_{s,t}| &= \left| \int_s^t \, \dif X_u \right| = \mathcal{O}(|t-s|^{\alpha}), \\
		\left|\int_s^t (X_u - X_s)\, \otimes \dif X_u \right| &= \left|\int_{s < u_1 < u_2 < t} \dif X_{u_1}\, \otimes \dif X_{u_2} \right|  =  \mathcal{O}(|t-s|^{2\alpha}).
	\end{align*}
	What is the regularity of higher order iterated integrals? We consider the third order first. We set  $\mathbb{X}^{(2)}_{s,t} = \int_s^t (X_u - X_s)\, \otimes \dif X_u$ and define
	\begin{align*}	
		\Xi_{u,v} \coloneqq \mathbb{X}^{(2)}_{0,u} \otimes \delta X_{u,v} + \delta X_{0,u} \otimes \mathbb{X}^{(2)}_{u,v} \in (\R^d)^{\otimes 3}.
	\end{align*}
	By assumption, $\| \Xi \|_{\alpha} < \infty$. Using the Chen identity, for $u<v<w$,
	\begin{align*}
		\delta \Xi_{u, v, w} &= (\mathbb{X}^{(2)}_{0,u} - \mathbb{X}^{(2)}_{0,v}) \otimes \delta X_{v,w} + \delta X_{0,u} \otimes (\mathbb{X}^{(2)}_{u,w} - \mathbb{X}^{(2)}_{u,v}) - \delta X_{0,v} \otimes \mathbb{X}^{(2)}_{v,w} \\
		&= -( \mathbb{X}^{(2)}_{u,v} + \delta X_{0,u} \otimes \delta X_{u,v})\otimes \delta X_{v,w} + \delta X_{0,u} \otimes(\mathbb{X}^{(2)}_{v,w} + \delta X_{u,v} \otimes \delta X_{v,w}) \\
		&\quad - \delta X_{0,v} \otimes \mathbb{X}^{(2)}_{v,w} \\
		&= -\mathbb{X}^{(2)}_{u,v} \otimes \delta X_{v,w} - \delta X_{u,v} \otimes \mathbb{X}^{(2)}_{v,w}.
	\end{align*}
	Therefore, $\| \delta \Xi \|_{3\alpha} < \infty$. From the Sewing lemma,
	\begin{align*}
		\delta \mathcal{I} \Xi_{s,t} = \lim_{|\mathcal{P}| \to 0} \sum_{[u,v] \in \mathcal{P}} \Xi_{u,v} &=  \lim_{|\mathcal{P}| \to 0} \sum_{[u,v] \in \mathcal{P}} \mathbb{X}^{(2)}_{0,u} \otimes \delta X_{u,v} + \delta X_{0,u} \otimes \mathbb{X}^{(2)}_{u,v}  \\
		&= \lim_{|\mathcal{P}| \to 0} \sum_{[u,v] \in \mathcal{P}} \mathbb{X}^{(2)}_{0,u} \otimes \delta X_{u,v} \\
		&= \int_s^t \mathbb{X}^{(2)}_{0,u}\, \otimes \dif X_u
	\end{align*}
	exists. In the second equality, we used that $|\mathbb{X}^{(2)}_{u,v}| = \mathcal{O}(|v-u|^{2\alpha})$ and $2\alpha > 1$. The Sewing lemma also tells us that
	\begin{align*}
		\left|\int_{s < u_1 < u_2 < u_3 < t} \dif X_{u_1}\, \otimes \dif X_{u_2} \otimes \dif X_{u_3} \right|  =  \mathcal{O}(|t-s|^{3\alpha}).
	\end{align*}
	Our goal is now to deduce the right regularity of iterated integrals of any order. As we will see in the sequel, we can repeat the previous argument by applying the Sewing lemma. Another property of iterated integrals we know from smooth functions is that their value decays very quickly when considering higher orders. To prove this property in our context, we need the \emph{neo-classical inequality} that we cite now.

	\begin{theorem}[Neo-classical inequality]
		For $\alpha\in (0,1]$, $n\in\mathbb{N}$ and $s,t>0$,
		\begin{align*}
			\alpha \sum_{0\leq j\leq n}\frac{s^{j\alpha}t^{(n-j)\alpha}}{(j\alpha)!((n-j)\alpha)!}\leq \frac{(t+s)^{n\alpha}}{(n\alpha)!},
		\end{align*}
		where $(j\alpha)! \coloneqq \Gamma(1+j\alpha)$ and 
		\begin{align*}
			\Gamma(z) = \int_{0}^{\infty}t^{z-1}\exp(-t)\, \mathrm{d} t, \ \   \ z\in\mathbb{C}\ \ \text{and}\ \  \mathcal{R}(z)>0 .
		\end{align*} 
	\end{theorem}
	
	\begin{proof}
		\cite{HH10}.
	\end{proof}

	We can now prove a first important result in rough paths theory, the Extension theorem.
	
	\begin{theorem}[Lyon's extension theorem]\label{thm:lyons_extension}
		Let $\mathbf{X} \colon \Delta \to T^N(\R^d)$ be a multiplicative functional with 
		\begin{align*}
			\vertiii{\mathbf{X}}_{\alpha} \coloneqq \max_{n = 1,\ldots,N} \sup_{0 \leq s < t \leq T} \frac{|\mathbb{X}^{(n)}_{s,t}|}{|t-s|^{n\alpha}} < \infty 	 
		\end{align*}
		for $N + 1> \frac{1}{\alpha}$. Then $\mathbf{X}$ has a unique extension to a multiplicative functional $\tilde{\mathbf{X}} \colon \Delta \to T((\R^d))$ with the same regularity. More precisely, $\tilde{\mathbf{X}}$ satisfies Chen's relation, $\tilde{\mathbb{X}}^{(n)} = \mathbb{X}^{(n)}$ for every $n = 0,\ldots,N$ and there are constants $M$ and $\beta$ such that
		\begin{align*}
			\sup_{0 \leq s < t \leq T} \frac{|\tilde{\mathbb{X}}^{(n)}_{s,t}|}{|t-s|^{n\alpha}} \leq \frac{M^n}{\beta(n\alpha)!}
		\end{align*}
		holds for every $n \in \N_0$.
	\end{theorem}

	\begin{proof}
		\textbf{Existence}: We use an induction argument. Let $n\geq N$, and assume for every $1\leq j\leq n$, $\mathbb{X}^{(j)}$ is well-defined and for some $M,\beta >0$ satisfying 
		\begin{align}\label{Induction}
			\sup_{0\leq s< t\leq T}\frac{|\mathbb{X}^{(j)}_{s,t}|}{|t-s|^{j\alpha}} \leq \frac{M^j}{\beta(j\alpha)!}\ .\ \ \ \ 
		\end{align}
		Also, assume for $2\leq k\leq n$
		\begin{align}\label{chen}
			\mathbb{X}^{(k)}_{s,t}=\mathbb{X}^{(k)}_{s,u}+\mathbb{X}^{(k)}_{u,t}+\sum_{1\leq i\leq k-1}\mathbb{X}^{(i)}_{s,u}\otimes \mathbb{X}^{(k-i)}_{u,t},\ \ \ \forall  s,u,t\in [0,T]: \  s\leq u\leq t.
		\end{align}
		First, we show one can apply the Sewing lemma to define the following integrals 
		\begin{align*}
			\Gamma^{n+1}(s,t):=\int_{s}^{t}\mathbb{X}^{n}_{0,\sigma}\otimes\mathrm{d}X_{\sigma}, \ \ \ \ X_{\sigma}:=\mathbb{X}_{0,\sigma}^{{1}}\ .
		\end{align*}
		Set 
		\begin{align*}
			\Xi^{(n+1)}_{s,t}=\sum_{1\leq j\leq n}\mathbb{X}^{(n+1-j)}_{0,s}\otimes\mathbb{X}^{(j)}_{s,t},
		\end{align*}
		then for $u<v<w$, from \eqref{chen} 
		\begin{align}\label{increment}
			\begin{split}
				&\delta\Xi^{(n+1)}_{u,v,w}=\Xi^{(n+1)}_{u,w}-\Xi^{(n+1)}_{u,v}-\Xi^{(n+1)}_{v,w}=\sum_{1\leq j\leq n}\mathbb{X}^{(n+1-j)}_{0,u}\otimes [\mathbb{X}^{(j)}_{u,w}-\mathbb{X}^{(j)}_{u,v}]-\sum_{1\leq j\leq n}\mathbb{X}^{(n+1-j)}_{0,v}\otimes \mathbb{X}^{(j)}_{v,w}\\&\quad  =\sum_{1\leq j\leq n}\mathbb{X}_{0,u}^{(n+1-j)}\otimes \mathbb{X}^{(j)}_{v,w}+\sum_{2\leq j\leq n}\sum_{1\leq i\leq j-1}\mathbb{X}^{(n+1-j)}_{0,u}\otimes [\mathbb{X}^{(j-i)}_{u,v}\otimes\mathbb{X}^{(i)}_{v,w}]-\sum_{1\leq j\leq n}\mathbb{X}^{(n+1-j)}_{0,u}\otimes \mathbb{X}^{(j)}_{v,w}\\&\quad -\sum_{1\leq j\leq n}\mathbb{X}^{(n+1-j)}_{u,v}\otimes \mathbb{X}^{(j)}_{v,w}-\sum_{1\leq j\leq n-1}\sum_{1\leq i\leq n-j}[\mathbb{X}^{(n+1-j-i)}_{0,u}\otimes\mathbb{X}^{(i)}_{u,v}\otimes ]\otimes \mathbb{X}^{(j)}_{v,w}\\
				&\quad = -\sum_{1\leq j\leq n}\mathbb{X}^{(n+1-j)}_{u,v}\otimes \mathbb{X}^{(j)}_{v,w}\  .
			\end{split}
		\end{align}
		From our induction assumption \eqref{Induction} and the neo-classical inequality,
		\begin{align}\label{induction_step}
			\begin{split}
				&|\delta\Xi^{(n+1)}_{u,v,w} |_{(n+1)\beta}\leq \sum_{1\leq j\leq n}|\mathbb{X}^{(n+1-j)}_{u,v}| |\mathbb{X}^{(j)}_{v,w}|\leq  \frac{M^{n+1}}{\beta^2}\sum_{1\leq j\leq n}\frac{(v-u)^{(n+1-j)\alpha}(w-v)^{j\alpha}}{(j\alpha)!((n+1-j)\alpha)!}\\&\quad \leq \frac{M^{n+1}(w-v)^{(n+1)\alpha}}{\alpha \beta^2((n+1)\alpha)!}.
			\end{split}
		\end{align}
		We finally define
		\begin{align}\label{induction_second}
			\mathbb{X}^{(n+1)}_{s,t}:=\delta\mathcal{I}\Xi_{s,t}^{(n+1)}-\Xi_{s,t}^{(n+1)},\ \ \  s,t\in [0,T].
		\end{align}
		Note that from \eqref{induction_step} and \eqref{eqn:sewing},
		\begin{align*}
			&|\mathbb{X}^{(n+1)}_{s,t}|\leq \big(1+2^{(n+1)\alpha}(\zeta((n+1)\alpha)-1)\big)(t-s)^{(n+1)\alpha}\Vert \delta\mathcal{I}\Xi^{(n+1)}\Vert\\ &\quad \leq \frac{ 2^{(N+1)\alpha}(\zeta((N+1)\alpha)-1)+1}{\alpha \beta^2}\times\frac{M^{n+1}}{((n+1)\alpha)!}.
		\end{align*}
		For $\beta$ satisfying 
		\begin{align*}
			\beta\ge \frac{2^{(N+1)\alpha}(\zeta((N+1)\alpha)-1)+1}{\alpha^2},
		\end{align*}
		we choose  $M>0$ such that
		\begin{align*}
			\forall j,\   1\leq j\leq N: \ \   \sup_{0\leq s< t\leq T}\frac{|\mathbb{X}^{(j)}_{s,t}|}{|t-s|^{j\alpha}}\leq \frac{M^{j}}{\beta(j\alpha)!}
		\end{align*}
		and our claim is proved. It only remains to prove that
		\begin{align*}
			\mathbb{X}^{(n+1)}_{s,t}=\mathbb{X}^{(n+1)}_{s,u}+\mathbb{X}^{(n+1)}_{u,t}+\sum_{1\leq j\leq n}\mathbb{X}^{(j)}_{s,u}\otimes \mathbb{X}^{(n+1-j)}_{u,t},\ \ \ \forall  s,u,t\in [0,T]: \  s\leq u\leq t,
		\end{align*}
		which again follows by an induction argument. Indeed, by \eqref{increment} and \eqref{induction_second}, 
		\begin{align*}
			\mathbb{X}^{(n+1)}_{s,t}-\mathbb{X}^{(n+1)}_{s,u}-\mathbb{X}^{(n+1)}_{u,t}=-\delta\Xi^{(n+1)}_{s,u,t}=\sum_{1\leq j\leq n}\mathbb{X}^{(n+1-j)}_{s,u}\otimes \mathbb{X}^{(j)}_{u,t} .
		\end{align*}
		\textbf{Uniqueness}: Assume that $\tilde{\mathbf{X}}$ and $\hat{\mathbf{X}}$ are two extensions of $\mathbf{X}$ that agree up to some level $n \geq N$. Set 
		\begin{align*}
			\Psi_{s,t} \coloneqq \tilde{\mathbb{X}}^{(n+1)}_{s,t} - \hat{\mathbb{X}}^{(n+1)}_{s,t}.
		\end{align*}
		From Chen's identity, for $s < u < t$,
		\begin{align*}
			\Psi_{s,t} &= \pi_{n+1}(\tilde{\mathbf{X}}_{s,u} \otimes \tilde{\mathbf{X}}_{u,t}) - \pi_{n+1}(\hat{\mathbf{X}}_{s,u} \otimes \hat{\mathbf{X}}_{u,t}) \\
			&= \tilde{\mathbb{X}}^{(n+1)}_{s,u} + \tilde{\mathbb{X}}^{(n+1)}_{u,t} - \hat{\mathbb{X}}^{(n+1)}_{s,u} -  \hat{\mathbb{X}}^{(n+1)}_{u,t} \\
			&= \Psi_{s,u} + \Psi_{u,t}. 
		\end{align*}
		It follows that $t \mapsto \Psi_t := \Psi_{0,t}$ is a path that has $(n+1)\alpha$-H\"older regularity. Since $(n+1)\alpha > 1$, $\Psi_t$ is constant, thus $\Psi_{s,t} = 0$ for every $s < t$ which shows $\tilde{\mathbb{X}}^{(n+1)}= \hat{\mathbb{X}}^{(n+1)}$. Therefore, our claim about the uniqueness is proved.
	\end{proof}
	% 
	%  \begin{remark}
		% As we observed during the proof, for some M, we can obtain the following bounds
		% 		\begin{align}\label{growth}
			% 		\forall n\geq 1: \ \ \ 	\| {\mathbb{X}}^{(n)} \|_{n \alpha} \leq\frac{M^{n}}{(n\alpha)!}.
			% \end{align}
		% \end{remark}
	The Extension theorem gives us the exact regularity of iterated Young integrals of any order. Moreover, it tells us that multiplicative functionals having a certain regularity up to a sufficiently high level can be uniquely extended to $T^{M}(\mathbb{R}^d)$ for any other integer $M>\frac{1}{\alpha}$. A rough path will be a multiplicative functional that has such an extension.

	\begin{definition}
		Let $\alpha \in (0,1]$. An \emph{$\alpha$-H\"older rough path} $\mathbf{X}$ is a multiplicative functional $\mathbf{X} \colon \Delta \to T^N(\R^d)$ such that
		\begin{align*}
			\vertiii{\mathbf{X}}_{\alpha} \coloneqq \max_{n = 1,\ldots,N} \sup_{0 \leq s < t \leq T} \frac{|\mathbb{X}^{(n)}_{s,t}|}{|t-s|^{n\alpha}}< \infty 
		\end{align*}
		where
		\begin{align*}
			\lfloor 1/\alpha \rfloor \coloneqq \max\left\{n \in \N\, :\, n \leq \frac{1}{\alpha} \right\}  = N,
		\end{align*}
		i.e. $N \leq \frac{1}{\alpha} < N+1$. The set of $\alpha$-H\"older rough paths is denoted by $\mathscr{C}^{\alpha}([0,T], \R^d)$ or simply by $\mathscr{C}^{\alpha}$. If $X \colon [0,T] \to \R^d$ is an $\alpha$-H\"older path and $\mathbf{X}$ a rough path with $\mathbb{X}^{(1)} = \delta X$, we call $\mathbf{X}$ a \emph{rough path lift} of $X$. If $\mathbf{X}$ is a rough path, the unique extension $\tilde{\mathbf{X}}$ provided by Theorem \ref{thm:lyons_extension} is called the \emph{Lyons lift} of $\mathbf{X}$.
	\end{definition}
	
	\subsection{Linear equations driven by a rough path}
	We now return to linear equations. In fact, we will see now that if $\mathbf{X}$ is an $\alpha$-H\"older rough path, we can solve linear equations driven by this path. We will first describe the equation we are looking at. Let $A_1,\ldots,A_d \in \R^{m \times m}$ and define a linear map $A \colon \R^d \to \R^{m \times m}$ by setting
	\begin{align*}
		Av \coloneqq A_1 v^1 + \ldots + A_d v^d,\qquad v = (v^1,\ldots,v^d) \in \R^d.
	\end{align*}
	Set
		\begin{align*}
			&\sigma:\mathbb{R}^m\rightarrow L(\mathbb{R}^d,\mathbb{R}^m),\\
			&\sigma(Z)(W):=(AW)Z,\ \ \ Z\in\mathbb{R}^{m} \  \text{and}\  \ W\in \mathbb{R}^{d}.
		\end{align*}
		We aim to solve
		\begin{align}\label{eqn:linear_RDE2}
			\begin{split}
				\dif Y_t &= \sigma(Y_t)\, \dif \mathbf{X}_t; \qquad t \in [0,T],\\
				Y_0 &= y \in \R^m.
			\end{split}
		\end{align}
	A natural candidate for a solution to \eqref{eqn:linear_RDE2} is
	\begin{align}\label{eqn:power_series}
		Y_t = \sum_{n = 0}^{\infty} A^{\otimes n} \left( \int_{0 < s_1 < \ldots < s_n < t} \dif X_{s_1} \otimes \cdots \otimes  \dif X_{s_n} \right)(y)
	\end{align}
	where we write $\int_{0 < s_1 < \ldots < s_n < t} \dif X_{s_1} \otimes \cdots \otimes  \dif X_{s_n} $ for the element $\mathbb{X}^{(n)}_{0,t}$ that is uniquely defined for every $n \in \N$ due to Lyons' Extension theorem. In the expression above, $A^{\otimes n} \colon (\R^d)^{\otimes n} \to \R^{m \times m}$ is the linear map defined by $A^{\otimes 0}(v) = I_m$ and
	\begin{align*}
		A^{\otimes n}(e_{i_1} \otimes \cdots \otimes e_{i_n}) = A_{i_1} \cdots A_{i_n}, \quad n \geq 1,
	\end{align*}
	where $\{e_1,\ldots,e_d\}$ denotes the Euclidean basis of $\R^d$. For example, if $A_1 = \cdots = A_d = I_m$, 
	\begin{align*}
		Y_t = y &+ y \sum_{i = 1}^d \int_0^t \dif X^i_s + \ldots  + y \sum_{i_1,\ldots,i_n = 1}^d \int_{0 < t_1 < \cdots < t_n < t} \dif X_{t_1}^{i_1} \cdots \dif X_{t_n}^{i_n} + \ldots
	\end{align*}
	where we use the notation $ \int_{0 < t_1 < \cdots < t_n < t} \dif X_{t_1}^{i_1} \cdots \dif X_{t_n}^{i_n} = \mathbb{X}_{0,t}^{(n);i_1,\ldots,i_n}$. Note that the infinite sum \eqref{eqn:power_series} indeed converges due to the superexponential decay of the iterated integrals $\mathbb{X}^{(n)}$ deduced in Theorem \ref{thm:lyons_extension}.

	\subsection{Brownian motion as a rough path}
	We saw that rough paths can be used to solve linear equations, or linear \emph{rough differential equations}. The natural question is now how we can use this result to solve linear stochastic differential equations pathwise. This would be possible if we could show that a given stochastic process can be ``naturally extended'' to a rough paths valued process. The most important process in stochastic analysis is the Brownian motion. Let $B = (B^1, \ldots, B^d)$ be a $d$-dimensional Brownian motion, i.e. the $B^i$, $i = 1, \ldots,d$, are independent, real valued Brownian motions. We know that the Brownian motion has trajectories that are $\alpha$-H\"older continuous for every $\alpha < \frac{1}{2}$. Therefore, we can construct a rough paths valued process if we determine the second iterated integral. There are (at least) two natural candidates: First, we can set $\mathbf{B}^{\text{It\=o}} = (1,B,\mathbb{B}^{\text{It\=o}})$ where
	\begin{align*}
		\mathbb{B}^{\text{It\=o}}_{s,t} = \int_s^t (B_u - B_s)\, \otimes \dif B_u.
	\end{align*}
	The integral is understood as an It\=o integral. Another choice would be $\mathbf{B}^{\text{Strat}} = (1,B,\mathbb{B}^{\text{Strat}})$,
	\begin{align*}
		\mathbb{B}^{\text{Strat}}_{s,t} = \int_s^t (B_u - B_s)\, \otimes  \circ\dif B_u = \mathbb{B}^{\text{It\=o}}_{s,t} + I_d \frac{(t-s)}{2}
	\end{align*}
	where the integral is understood as Stratonovich integral.  Since both the It\=o and the Stratonovich integral satisfy $\int_s^t = \int_s^u + \int_u^t$ for $s < u < t$, they  satisfy Chen's relation, thus they are multiplicative functionals almost surely. It remains to check that also the iterated integrals have the right H\"older-regularity. To prove this regularity, we first state the following version of the Kolmogorov-Chentsov theorem:

	\begin{theorem}[Kolmogorov-Chentsov theorem for multiplicative functionals]\label{thm:kolmogorov}
		Let $\mathbf{X} \colon \Delta \to T^2(\R^d)$ be a random continuous multiplicative functional, $q \geq 2$, $\beta > \frac{1}{q}$ and assume that
		\begin{align*}
			\| \mathbb{X}^{(1)}_{s,t} \|_{L^q} \leq C|t-s|^{\beta} \quad \text{and} \quad \| \mathbb{X}^{(2)}_{s,t} \|_{L^{q/2}} \leq C|t-s|^{2 \beta}
		\end{align*}
		for a constant $C > 0$ and any $s,t \in [0,T]$. Then for all $\alpha \in [0, \beta - 1/q)$, there are random variables $K^1_{\alpha} \in L^q$ and $K^2_{\alpha} \in L^{\frac{q}{2}}$ such that
		\begin{align*}
			|\mathbb{X}^{(1)}_{s,t}| \leq K^1_{\alpha} |t-s|^{\alpha} \quad \text{and} \quad |\mathbb{X}^{(2)}_{s,t}| \leq K^2_{\alpha} |t-s|^{2 \alpha}
		\end{align*}
		for all $s,t \in [0,T]$. In particular, $\vertiii{\mathbf{X}}_{\alpha} < \infty$ almost surely.			
	\end{theorem}
	
	\begin{proof}
		The proof is similar to the one of the classical Kolmogorov-Chentsov theorem that we already saw in Theorem \ref{thm:kolmogorov1}. We proceed as in \cite[Theorem 3.1]{FH20}. We assume that $T = 1$. Defining $D_n$ and $D$ as in the proof of Theorem \ref{thm:kolmogorov1}, we set 
		\begin{align*}
			K_n \coloneqq \sup_{t \in D_n} |\mathbb{X}^{(1)}_{t, t+ 2^{-n}}| \quad \text{and} \quad \tilde{K}_n \coloneqq \sup_{t \in D_n} |\mathbb{X}^{(2)}_{t, t+ 2^{-n}}|.
		\end{align*}
		As before, one can check that $\E(K^q_n) \leq C^q |D_n|^{q\beta -1}$ and  $\E(\tilde{K}^{q/2}_n) \leq C^{q/2} |D_n|^{q\beta -1}$. Fix $s, t \in D$ and choose $m$ with $|D_{m+1}| < |t-s| \leq |D_m|$. Furthermore, choose
		\begin{align*}
			s = \tau_0 < \tau_1 < \ldots < \tau_N = t
		\end{align*}
		as in the proof of Theorem \ref{thm:kolmogorov1}. Then,
		\begin{align*}
			|\mathbb{X}^{(1)}_{s,t}| \leq \max_{0 \leq i \leq N-1} |\mathbb{X}^{(1)}_{s,\tau_{i+1}}| \leq \sum_{i = 0}^{N-1} |\mathbb{X}^{(1)}_{\tau_i,\tau_{i+1}}| \leq 2 \sum_{n \geq m+1} K_n.
		\end{align*}
		Using the Chen relation repeatedly gives
		\begin{align*}
			|\mathbb{X}^{(2)}_{s,t}| &= \left| \sum_{i = 0}^{N-1} \mathbb{X}^{(2)}_{\tau_i,\tau_{i+1}} + \mathbb{X}^{(1)}_{s,\tau_i} \otimes \mathbb{X}^{(1)}_{\tau_i,\tau_{i+1}} \right| 
			\leq \sum_{i = 0}^{N-1} | \mathbb{X}^{(2)}_{\tau_i,\tau_{i+1}}|  + |\mathbb{X}^{(1)}_{s,\tau_i}| | \mathbb{X}^{(1)}_{\tau_i,\tau_{i+1}} | \\
			&\leq \sum_{i = 0}^{N-1} | \mathbb{X}^{(2)}_{\tau_i,\tau_{i+1}}|  +  \max_{0 \leq i \leq N-1} |\mathbb{X}^{(1)}_{s,\tau_{i+1}}| \sum_{i = 0}^{N-1} | \mathbb{X}^{(1)}_{\tau_i,\tau_{i+1}} | 
			\leq 2  \sum_{n \geq m+1} \tilde{K}_n + \left( 2 \sum_{n \geq m+1} K_n \right)^2.
		\end{align*}
		In the proof of Theorem \ref{thm:kolmogorov1}, we have already seen that this implies that for every $s < t$
		\begin{align*}
			\frac{|\mathbb{X}^{(1)}_{s,t}|}{|t-s|^{\alpha}} \leq 2 \sum_{n=0}^{\infty} \frac{K_n}{|D_n|^{\alpha}} \eqqcolon K^1_{\alpha}
		\end{align*}
		and $K^1_{\alpha} \in L^q$. Similarly, 
		\begin{align*}
			\frac{|\mathbb{X}^{(2)}_{s,t}|}{|t-s|^{2 \alpha}} \leq 2 \sum_{n=0}^{\infty} \frac{\tilde{K}_n}{|D_n|^{2 \alpha}} + \left(2 \sum_{n=0}^{\infty} \frac{K_n}{|D_n|^{\alpha}} \right)^2 = K^2_{\alpha} + (K^1_{\alpha})^2 
		\end{align*}
		where
		\begin{align*}
			K^2_{\alpha} \coloneqq \sum_{n=0}^{\infty} \frac{\tilde{K}_n}{|D_n|^{2 \alpha}}.
		\end{align*}
		It is then straightforward to check that $K^2_{\alpha} \in L^{\frac{q}{2}}$ which finishes the proof.

	\end{proof}
	
	\begin{corollary}\label{cor:Bm_lift}
		We have $\vertiii{\mathbf{B}^{\text{It\=o}}}_{\alpha} < \infty$ and  $\vertiii{\mathbf{B}^{\text{Strat}}}_{\alpha} < \infty$ almost surely for every $\alpha < \frac{1}{2}$.
	\end{corollary}
	
	\begin{proof}
		Using Brownian scaling, one can show that the conditions of Theorem \ref{thm:kolmogorov} hold for the It\=o- and for the Stratonovich lift of the Brownian motion for $\beta = \frac{1}{2}$ and every $q \geq 2$. 
	\end{proof}
	
	From Corollary \eqref{cor:Bm_lift}, we know that $\mathbf{B}^{\text{It\=o}}$ and $\mathbf{B}^{\text{Strat}}$ are both rough path valued stochastic processes. Therefore, we can use them both to solve linear stochastic differential equations driven by a Brownian motion. However, choosing the It\=o or the Stratonovich rough path lift leads to different solutions, which is natural and well known in stochastic analysis. In fact, the choice of the rough path lift should be regarded as another parameter in the equation and depends on the problem one aims to find a model for.

	\section{The space of rough paths}
	
	\subsection{Metrics on rough paths spaces and separability}
	In the previous section, we defined the set $\mathscr{C}^{\alpha}([0,T],\R^d)$ of $\alpha$-H\"older rough paths. Note that there is no meaningful notion of the sum of two rough paths, i.e. $\mathscr{C}^{\alpha}$ is not a linear space. We will see now that it is still a metric space.
	\begin{definition}
		Let $\mathbf{X}, \mathbf{Y} \in \mathscr{C}^{\alpha}$. Then we define
		\begin{align*}
			\varrho_{\alpha}(\mathbf{X},\mathbf{Y}) \coloneqq \sum_{n = 1}^{\lfloor 1/ \alpha \rfloor} \sup_{0 \leq s < t \leq T} \frac{|  \mathbb{X}^{(n)}_{s,t} - \mathbb{Y}^{(n)}_{s,t}|}{|t-s|^{n\alpha}}.
		\end{align*}
		
	\end{definition}
	It is not hard to see that $\varrho_{\alpha}$ is a metric on $\mathscr{C}^{\alpha}$. Moreover, one can prove the following:
	
	\begin{proposition}
		For every $\alpha \in (0,1]$, the space $(\mathscr{C}^{\alpha} , \varrho_{\alpha})$ is a complete metric space.
	\end{proposition}
	
	\begin{proof}
		The arguments are the same as those used for proving that the usual H\"older spaces are complete. We leave the details to the reader. A detailed proof can be found in \cite[Lemma 3.3.3]{LQ02}. 
	\end{proof}
	
	Sometimes, it is desirable to work with separable rough paths spaces. However, since H\"older spaces are not separable, we cannot expect that the spaces $\mathscr{C}^{\alpha}$ are separable. To solve this issue for H\"older spaces, one often considers \emph{little H\"older spaces} that are defined as the closure of the space of smooth functions in the $\alpha$-H\"older metric. A similar definition works for rough paths spaces, too.
	\begin{definition}
		Let $X \colon [0,T] \to \R^d$ be smooth (e.g. piecewise continuously differentiable) and $\alpha \in (0,1]$. Then we call $\mathbb{X} \in \mathscr{C}^{\alpha}$ with
		\begin{align*}
			\mathbb{X}_{s,t}^{(n)} = \int_{s < u_1 < \ldots < u_n < t} \dif X_{u_1} \otimes \cdots \otimes \dif X_{u_n}
		\end{align*}
		the \emph{canonical lift} of $X$ to an $\alpha$-H\"older rough path. Rough paths $\mathbf{X} \in \mathscr{C}^{\alpha}$ of this form are also called \emph{smooth rough paths}. The space $\mathscr{C}^{\alpha}_g$ is defined as the closure of smooth rough paths in the metric $\varrho_{\alpha}$. The elements in  $\mathscr{C}^{\alpha}_g$ are called \emph{geometric rough paths}.
	\end{definition}

	\begin{proposition}
		For every $\alpha \in (0,1]$, the space $(\mathscr{C}_g^{\alpha} , \varrho_{\alpha})$ is a complete separable metric, i.e. \emph{Polish} space.
	\end{proposition}
	
	\begin{proof}
		Completeness follows by definition. The idea to show separability is to find a complete separable space of smooth paths containing all piecewise $\mathcal{C}^1$-paths for which the canonical lift map is continuous. An example is the space obtained by taking the closure of arbitrarily often differentiable paths with respect to the total variation distance. Details can be found in \cite[Appendix A and B]{BRS17}.
	\end{proof}
	
	\begin{proposition}\label{geometric_brownian}
		The process $\mathbf{B}^{\mathrm{Strat}} = (1,B , \mathbb{B}^{\mathrm{Strat}}) $ takes values in the space $\mathscr{C}^{\alpha}_g$ for every $\frac{1}{3} < \alpha < \frac{1}{2}$ almost surely.
	\end{proposition}

	\begin{proof}
		For simplicity, $T = 1$. Choose $\alpha'$ such that $\alpha < \alpha' < \frac{1}{2}$. We know that
		\begin{align*}
			\vertiii{\mathbf{B}^{\text{Strat}}}_{\alpha'} < \infty.
		\end{align*}
		% Let $\mathbb{B} \coloneqq \mathbb{B}^{\mathrm{Strat}}$. 
		For $n \in \N$, we define $B(n)$ to be the piecewise-linear approximation of $B$ at the dyadic points $0 < 2^{-n} < \cdots < (2^n -1)2^{-n} < 1$, i.e.
		\begin{align*}
			B_t(n) = B_{k2^{-n}} + 2^n (t - k2^{-n}) (B_{(k+1)2^{-n}} - B_{k2^{-n}}), \qquad t \in [k2^{-n}, (k+1)2^{-n}].
		\end{align*}
		Let $\mathbf{B}(n)$ be the canonical lift of $B(n)$ to an $\alpha$-H\"older rough path. With some basic calculations, one can show that 
		\begin{align*}
			\| \delta B(n)_{s,t} \|_{L^2} \leq C|t-s|^{\frac{1}{2}} \quad \text{and} \quad \| \mathbb{B}(n)_{s,t} \|_{L^2} \leq C|t-s|
		\end{align*}
		holds for every $s<t$ for a constant $C$ that is independent of $n$. Since $B$ is Gaussian, the same estimates also hold for the $L^q$-norm for every $q \geq 2$. The Kolmogorov-Chentsov theorem for multiplicative functionals implies that 
		\begin{align*}
			\sup_{n \in \N} \vertiii{\mathbf{B}(n)}_{\alpha'} < \infty.
		\end{align*}
		a.s. To prove that $\varrho_{\alpha}(\mathbf{B}^{\text{Strat}},\mathbf{B}(n)) \to 0$, by the Arzel\`a-Ascoli theorem, it is sufficient to show that $B(n) \to B$ and $\mathbb{B}(n) \to \mathbb{B}^{\text{Strat}}$ pointwise as $n \to \infty$. The first statement is clear. For the second, we first note that
		\begin{align*}
			\int_s^t (B^i_u(n) - B^i_s(n))\, \dif B^i_u(n) = \frac{(B_t(n) - B_s(n))^2}{2} \to \frac{(B_t - B_s)^2}{2} = \int_s^t (B^i_u - B^i_s)\, \circ \dif B^i_u.
		\end{align*}
		Define
		\begin{align*}
			\mathcal{F}_n \coloneqq \sigma(B_k\, :\, k \in \{ 0,2^{-n},\ldots,(2^n-1)2^{-n},1\}).
		\end{align*}
		Then $(\mathcal{F}_n)_{n \geq 1}$ is a filtration. Fix $t \in [0,T]$. By Gaussian conditioning, one can check that $B_t(n) = \E[B_t\, |\, \mathcal{F}_n]$. From the martingale convergence theorem, it follows that
		\begin{align*}
			B_t(n) = \E[B_t\, |\, \mathcal{F}_n]  \to B_t
		\end{align*}
		almost surely and in $L^p$ for any $p \geq 1$ as $n \to \infty$ (which yields an alternative proof of what we already know). For $i \neq j$,
		\begin{align*}
			\E \left( \int_0^t B^i_s\, \dif B^j_s \, |\, \mathcal{F}_n \right) &= \lim_{|\mathcal{P}| \to 0} \sum_{[u,v] \in \mathcal{P}} \E\left( B^i_u \delta B^j_{u,v} \, |\, \mathcal{F}_n \right) = \sum_{[u,v] \in \mathcal{P}} B^i_u(n) \delta B^j_{u,v}(n) \\
			&= \int_0^t B^i_s(n)\, \dif B^j_s(n).
		\end{align*}
		Therefore, the martingale convergence theorem yields that
		\begin{align*}
			\int_0^t B^i_s(n)\, \dif B^j_s(n) = \E \left( \int_0^t B^i_s\, \dif B^j_s \, |\, \mathcal{F}_n \right) \to \int_0^t B^i_s\, \dif B^j_s
		\end{align*}
		almost surely and in $L^p$ for any $p \geq 1$ as $n \to \infty$, which finishes the proof.
	\end{proof}
	
	A natural question is whether $\mathbf{B}^{\text{It\=o}}$ has geometric rough paths trajectories, too. We will see in the next section that this is not the case.

	\subsection{Shuffles and the signature}
	There is also an important algebraic property satisfied by geometric rough paths that is inherited from smooth rough paths. In fact, multiplying two iterated integrals of smooth paths yields a linear combination of iterated integrals. For example,
	\begin{align*}
		\int_0^T \dif X_s \cdot \int_0^T \dif Y_s = \int_{0 < s_1 < s_2 < T} \dif X_{s_1} \, \dif Y_{s_2} + \int_{0 < s_1 < s_2 < T} \dif Y_{s_1} \, \dif X_{s_2}.
	\end{align*}
	Note that this is a property that does not hold for every rough path. For example, the It\=o integral satisfies the identity
	\begin{align*}
		\int_0^T B_t \, \dif B_t = \frac{B^2_t}{2} - \frac{t}{2}
	\end{align*}
	which shows that the sample paths of $\mathbf{B}^{\text{It\=o}}$ behave differently.
	
	We aim to give a more detailed description of the product of iterated integrals of smooth paths. To do this, we introduce some more notation.
	\begin{definition}
		The direct sum
		\begin{align*}
			T(\R^d) \coloneqq \R \oplus \R^d \oplus (\R^d \otimes \R^d) \oplus \cdots = \bigoplus_{n = 0}^{\infty} (\R^d)^{\otimes n}
		\end{align*}
		is called \emph{tensor algebra}.
	\end{definition}
	One can show that the extended tensor algebra is the (algebraic) dual of the tensor algebra. We will identify the basis elements $e_{i_1} \otimes \cdots \otimes e_{i_n}$ in the tensor algebra $T(\R^d)$ with the words ${i_1 \cdots i_n}$ composed by the letters $1, \ldots, d$. The empty word will be denoted by ${\epsilon}$. For two words, we can define their shuffle product:
	\begin{definition}
		Let ${u}$, ${v}$ be words and ${a}$, ${b}$ be letters. The \emph{shuffle product} is defined recursively by
		\begin{align*}
			u \shuffle \epsilon &=  \epsilon \shuffle u = u, \\ \\
			ua \shuffle vb &= (u \shuffle vb)a + (ua \shuffle v) b.
		\end{align*}
		The shuffle product is extended bilinearly to a product
		\begin{align*}
			\shuffle \colon T(\R^d) \times T(\R^d) \to T(\R^d).
		\end{align*}
	\end{definition}

	\begin{example}
		\begin{enumerate}
			\item For example,
			\begin{align*}
				{12} \shuffle {3} &= {123} + {132} + {312}, \\
				{12} \shuffle {24} &= 2 \cdot {1224} + {1242} + {2124} + {2142}  + {2412}.
			\end{align*}
			
			\item Let $X \colon [0,T] \to \R^d$ be a smooth path (e.g. $\mathcal{C}^1$) and
			\begin{align}\label{eqn:canonical_lift}
				\mathbf{X}_{0,T} := \left( 1,\int_0^T \dif X_s, \ldots, \int_{0 < s_1 < \cdots < s_n < T} \dif X_{s_1} \otimes \cdots \otimes \dif X_{s_n} , \ldots \right) \in T((\R^d)).
			\end{align}
			With the notation we introduced above, we have, for example,
			\begin{align*}
				\langle {121} , \mathbf{X}_{0,T} \rangle &= \int_{0 < s_1 < s_2 < s_3 < T} \dif X^1_{s_1} \, \dif X^2_{s_2} \, \dif X^1_{s_3}, \\
				\langle \sqrt{3} \cdot {12} - 2 \cdot {21} , \mathbf{X}_{0,T} \rangle &= \sqrt{3} \int_{0 < s_1 < s_2 < T} \dif X^1_{s_1} \, \dif X^2_{s_2} - 2 \int_{0 < s_1 < s_2 < T} \dif X^2_{s_1} \, \dif X^1_{s_2} .
			\end{align*}
		\end{enumerate}

	\end{example}
	
	%\begin{center}
	% \includegraphics[height=100pt]{Riffle_shuffle.jpg}
	%\end{center}

	The main observation is the following:

	\begin{theorem}
		For $\mathbf{X}$ defined as in \eqref{eqn:canonical_lift}, for every $l_1, l_2 \in T(\R^d)$,
		\begin{align*}
			\langle l_1, \mathbf{X}_{0,T} \rangle \langle l_2, \mathbf{X}_{0,T} \rangle = \langle l_1 \shuffle l_2 , \mathbf{X}_{0,T} \rangle.
		\end{align*}	
	\end{theorem}
	
	\begin{proof}
		Let $u$ and $v$ be words and $a$ and $b$ be letters from the alphabet $\{1,\ldots,d\}$. The proof is by induction over the length of the words. Using the induction hypothesis, we have
		\begin{align*}
			\langle ua, \mathbf{X}_{0,T} \rangle \langle vb, \mathbf{X}_{0,T} \rangle &= \int_0^T \langle u, \mathbf{X}_{0,t} \rangle \, \dif X^a_t \cdot \int_0^T \langle v, \mathbf{X}_{0,s} \rangle \, \dif X^b_s \\
			&= \int_{0 < s,t < T} \langle u, \mathbf{X}_{0,t} \rangle \langle v, \mathbf{X}_{0,s} \rangle \, \dif X^a_t\, \dif X^b_s \\
			&= \int_{0 < t < s < T} \langle u, \mathbf{X}_{0,t} \rangle \langle v, \mathbf{X}_{0,s} \rangle \, \dif X^a_t\, \dif X^b_s + \int_{0 < s < t < T}   \langle u, \mathbf{X}_{0,t} \rangle \langle v, \mathbf{X}_{0,s} \rangle \, \dif X^b_s \, \dif X^a_t \\
			&= \int_0^T \langle ua, \mathbf{X}_{0,s} \rangle \langle v, \mathbf{X}_{0,s} \rangle \, \dif X^b_s + \int_0^T \langle u, \mathbf{X}_{0,t} \rangle \langle vb, \mathbf{X}_{0,t} \rangle \, \dif X^a_t \\
			&= \int_0^T \langle ua \shuffle v, \mathbf{X}_{0,s} \rangle  \, \dif X^b_s + \int_0^T \langle u \shuffle vb, \mathbf{X}_{0,t} \rangle \, \dif X^a_t \\ 
			&= \langle (ua \shuffle v)b, \mathbf{X}_{0,T} \rangle + \langle (u \shuffle vb)a, \mathbf{X}_{0,T} \rangle \\
			&= \langle (ua \shuffle v)b + (u \shuffle vb)a, \mathbf{X}_{0,T} \rangle \\
			&= \langle ua \shuffle vb, \mathbf{X}_{0,T} \rangle.
		\end{align*}

	\end{proof}
	
	The following corollary is immediate.
	
	\begin{corollary}\label{cor:shuffle_prop}
		Let $\mathbf{X} \in \mathscr{C}^{\alpha}_g$ be a geometric rough path. We identify $\mathbf{X}$ with its Lyons-lift to a path with values in $T((\R^d))$. Then for every $l_1, l_2 \in T(\R^d)$,
		\begin{align*}
			\langle l_1, \mathbf{X}_{0,T} \rangle \langle l_2, \mathbf{X}_{0,T} \rangle = \langle l_1 \shuffle l_2 , \mathbf{X}_{0,T} \rangle.
		\end{align*}
	\end{corollary}

	\begin{remark}\label{remark:signature}
		Let $\mathbf{X}$ be a geometric rough path. As usual, we identify $\mathbf{X}$ with its Lyons lift to a path with values in $T((\R^d))$. Then the element $\mathbf{X}_{0,T} \in T((\R^d))$ is called the \emph{signature} of the rough path $\mathbf{X}$. The signature is important since it contains all (necessary) information about the rough path. Indeed, in a series of papers, it was shown that the signature determines a geometric rough path completely up to so-called ``tree-like'' excursions \cite{Che58, HL10, BGLY16}. If $\mathbf{X}$ is random, the \emph{expected signature} determines the law of $\mathbf{X}$ and can be seen as a Laplace transform for measures on path spaces \cite{CL16, CO22}. The (truncated) signature also plays an important role in machine learning as a way to extract characteristic features from a data stream, cf. \cite{CK16, LM22} for an overview. 
	\end{remark}
	
	\section{Controlled paths and rough integral}
	We aim to solve non-linear rough differential equations of the form
	\begin{align*}
		\dif Y_t &= \sigma(Y_t)\, \dif \mathbf{X}_t; \quad t \in [0,T] \\
		Y_0 &= y \in \R^m.
	\end{align*}
	As for the Young case, we want to interpret the equation as an integral equation:
	\begin{align*}
		Y_t = y + \int_0^t \sigma(Y_s) \, \dif \mathbf{X}_s; \quad t \in [0,T].
	\end{align*}
	We want to find a notion of an integral that coincides with the Young integral in case the integrand is smooth. That is, for a smooth function $f$ and a Brownian motion $B$, we would like to have that
	\begin{align*}
		\int_0^T f(s) \, \dif \mathbf{B}_s =  \int_0^T f(s) \, \dif B_s.
	\end{align*}
	Here, $\mathbf{B}$ may either denote $\mathbf{B}^{\text{It\=o}}$ or $\mathbf{B}^{\text{Strat}}$. If we want to perform a fixed point argument to solve the equation, it is desirable to look for a Banach space $E$ containing smooth functions such that the map
	\begin{align*}
		f \mapsto \big(t \mapsto \int_0^t f(s)\, \dif \mathbf{B}_s \big)
	\end{align*}
	is a continuous map from $E$ to itself. A minimal requirement for $E$ would be that it contains the trajectories of the Brownian motion, otherwise we would not be able to integrate constant functions. However, one can show that such a space $E$ does not exist:

	\begin{theorem}
		There is no space of functions $E$ carrying the Wiener measure on which we can define a continuous map $I \colon E \to E$ that coincides with the pathwise defined integral 
		\begin{align*}
			I(f) = \big(t \mapsto \int_0^t f(s)\, \dif B_s \big)
		\end{align*}
		for smooth functions $f$ on a set of full measure.
	\end{theorem}

	\begin{proof}
		Same idea as in the proof of Theorem \ref{Lyons}.
	\end{proof}
	The solution to this issue proposed by rough paths theory is that we allow the space $E$ to depend on the trajectory of the Brownian motion, i.e. we will define spaces $\{E_{\omega}\}_{\omega \in \Omega}$ for which $B(\omega) \in E_{\omega}$ with the property that 
	\begin{align*}
		E_{\omega} \ni f \mapsto \big( t \mapsto \int_0^t f(s)\, \dif \mathbf{B}_s(\omega) \big) \in E_{\omega}
	\end{align*}
	extends the integral map on smooth paths and is continuous.  Our goal is to define a ``rough integral'' of the form
	\begin{align*}
		\int_0^T Y_t \, \dif \mathbf{X}_t,
	\end{align*}
	for a given rough path $\mathbf{X} \in \mathscr{C}^{\alpha}$. Before moving forward, let us assume the following assumptions:

	\begin{assumption}
		For the sake of simplicity, we will assume  $\alpha \in (1/3,1/2]$ from now on.
	\end{assumption}
	Remember that we deduced the regularity of a 3-times iterated Young integral by introducing a ``compensator'':
	\begin{align*}
		\int_0^T \mathbb{X}_{0,t}^{(2)} \, \dif X_t = \sum_{|\mathcal{P}| \to 0} \sum_{[u,v] \in \mathcal{P}}  \mathbb{X}_{0,u}^{(2)} \otimes \delta X_{u,v} + \delta X_{0,u} \otimes \mathbb{X}_{u,v}^{(2)}. 
	\end{align*}
	This motivates the following ansatz: for given $\mathbf{X} = (1,X,\mathbb{X}) \in  \mathscr{C}^{\alpha}$ and $Y \colon [0,T] \to L(\R^d,\R^m)$, we assume that there exists a path $Y' \colon [0,T] \to L(\R^d \otimes \R^d, \R^m)$ for which we can define the limit
	\begin{align*}
		\int_0^T Y_t \, \dif \mathbf{X}_t = \sum_{|\mathcal{P}| \to 0} \sum_{[u,v] \in \mathcal{P}}  Y_u \delta X_{u,v} + Y'_u \mathbb{X}_{u,v}.
	\end{align*}
	As before, we will use the Sewing lemma to prove the existence of the limit. Set
	\begin{align*}
		\Xi_{u,v} \coloneqq Y_u \delta X_{u,v} + Y'_u \mathbb{X}_{u,v}.
	\end{align*}
	Clearly, $\| \Xi  \|_{\alpha} < \infty$. We have to make sure that $\| \delta \Xi \|_{\beta} < \infty$ for some $\beta > 1$. After some lines of calculations, we see that
	\begin{align*}
		\delta \Xi_{s,u,t} = -(\delta Y_{s,u} - Y'_s \delta X_{s,u}) X_{u,t} - \delta Y'_{s,u} \mathbb{X}_{u,t}.
	\end{align*}
	Therefore, we arrive at the conditions
	\begin{align*}
		| \delta Y'_{s,t}| &= \mathcal{O}(|t-s|^{\gamma}) \quad \text{where} \quad \quad \gamma + 2\alpha > 1 \quad \text{and} \\
		\quad| \delta Y_{s,t} - Y'_s \delta X_{s,t}| &= \mathcal{O}(|t-s|^{\tilde{\gamma}}) \quad \text{with} \quad \tilde{\gamma} + \alpha > 1.
	\end{align*}
	These conditions are in particular satisfied for $\gamma =\alpha$ and $\tilde{\gamma} = 2\alpha$. This observation motivates the following definition that was introduced by Gubinelli in \cite{Gub04} first.

	\begin{definition}
		Let $\mathbf{X} \in \mathscr{C}^{\alpha}([0,T], \R^d)$, $\alpha \in (1/3,1/2]$. A path $Y \in \mathcal{C}^{\alpha}([0,T], W)$ is said to be \emph{controlled by $\mathbf{X}$} if there exists a path $Y' \in \mathcal{C}^{\alpha}([0,T], L(\R^d, W))$ such that the remainder term $R^Y$ given by
		\begin{align*}
			R^Y_{s,t} \coloneqq \delta Y_{s,t} - Y'_s \delta X_{s,t}
		\end{align*}
		satisfies $\| R^Y \|_{2 \alpha} < \infty$. The path $Y'$ is called a \emph{Gubinelli-derivative} of $Y$. The set of all controlled paths $(Y,Y')$ is denoted by $\mathscr{D}^{\alpha}_X([0,T],W)$. If $(Y,Y') \in \mathscr{D}^{\alpha}_X([0,T],W)$, we set
		\begin{align*}
			\|Y,Y'\|_{X, \alpha} \coloneqq \|Y'\|_{\alpha} + \| R^Y \|_{2 \alpha}.
		\end{align*}
	\end{definition}
	
	\begin{example}
		\begin{enumerate}
			\item If $\mathbf{X} = (1,\delta X, \mathbb{X}^{(2)}) \in  \mathscr{C}^{\alpha}$, the path $X$ is controlled by $\mathbf{X}$. A Gubinelli-derivative is given by the constant function $Y' = I_d$.
			\item If $Y$ is smooth or, more precisely, $2 \alpha$-H\"older continuous, the path is controlled by $\mathbf{X}$ with Gubinelli-derivative $Y' = 0$.
		\end{enumerate}
		
	\end{example}
	
	It is easily seen that the space of controlled paths is a linear space for every fixed rough path $\mathbf{X}\in  \mathscr{C}^{\alpha}$. Moreover, one can prove the following: 
	\begin{proposition}
		The spaces $\mathscr{D}^{\alpha}_X([0,T],W)$ are Banach spaces with a norm given by
		\begin{align*}
			(Y,Y') \mapsto |Y_0| + |Y'_0| + \|Y,Y'\|_{X,\alpha} \eqqcolon \vertiii{Y,Y'}_{X,\alpha}.
		\end{align*}
		
	\end{proposition}
	
	\begin{proof}
		Straightforward.
	\end{proof}
	
	\begin{remark}
		Gubinelli derivatives are not unique, in general. Indeed, if $Y$ is smooth, we can choose $Y' = 0$. But if $X$ is smooth, too, we can in fact choose any $\mathcal{C}^{\alpha}$-path as a Gubinelli-derivative. On the contrary, if $X$ is not smooth, one can show uniqueness of the Gubinelli-derivative, cf. \cite[Proposition 6.4]{FH20}.
	\end{remark}
	The most important fact about controlled paths is that they are good integrands.
	
	\begin{theorem}\label{thm:rough_integration}
		Let $\mathbf{X} \in \mathscr{C}^{\alpha}([0,T], \R^d)$ and $(Y,Y') \in \mathscr{D}^{\alpha}_X([0,T],L(\R^d,\R^m))$. 
		\begin{enumerate}
			\item  The integral
			\begin{align*}
				\int_s^t Y_u\, \dif \mathbf{X}_u := \lim_{|\mathcal{P}| \to 0} \sum_{[u,v] \in \mathcal{P}} Y_u \delta X_{u,v} + Y'_u \mathbb{X}_{u,v}
			\end{align*}
			exists and satisfies the bound
			\begin{align}\label{eqn:bound_rough_integral}
				\left| \int_s^t Y_u\, \dif \mathbf{X}_u - Y_s \delta X_{s,t} - Y'_s \mathbb{X}_{s,t} \right| \leq C ( \|X\|_{\alpha} \|R^Y\|_{2 \alpha} + \| \mathbb{X} \|_{2\alpha} \|Y\|_{\alpha})|t-s|^{3 \alpha}.
			\end{align}
			\item The path $t \mapsto \int_0^t Y_u\, \dif \mathbf{X}_u$ is a controlled path with Gubinelli-derivative $Y$. The map 
			\begin{align*}
				(Y,Y') \mapsto \left(\int_0^{\cdot} Y_u\, \dif \mathbf{X}_u, Y \right) \eqqcolon (Z,Z')
			\end{align*}
			is a continuous linear map from $\mathscr{D}^{\alpha}_X([0,T],L(\R^d,\R^m))$ to $\mathscr{D}^{\alpha}_X([0,T],\R^m)$. Moreover, we have the bound
			\begin{align*}
				\| Z,Z' \|_{X,\alpha} \leq \|Y\|_{\alpha} + \|Y'\|_{\infty} \|\mathbb{X}\|_{2\alpha} + CT^{\alpha}(\|X\|_{\alpha} \|R^Y\|_{2\alpha} + \|\mathbb{X} \|_{2 \alpha} \|Y'\|_{\alpha}).
			\end{align*}
			
		\end{enumerate}

	\end{theorem}
	
	\begin{proof}
		As already indicated above, we use the Sewing lemma with 
		\begin{align*}
			\Xi_{u,v} \coloneqq Y_u \delta X_{u,v} + Y'_u \mathbb{X}_{u,v}.
		\end{align*}
		From
		\begin{align*}
			\delta \Xi_{s,u,t} = -R^Y_{s,u} X_{u,t} - \delta Y'_{s,u} \mathbb{X}_{u,t},
		\end{align*}
		we see that
		\begin{align*}
			\| \delta \Xi \|_{3\alpha} \leq \|R^Y\|_{2 \alpha} \|X\|_{\alpha} + \|Y'\|_{\alpha} \|\mathbb{X}\|_{2\alpha}
		\end{align*}
		and the first assertion follows. For the second assertion, we have to prove that
		\begin{align*}
			R^Z_{s,t} \coloneqq \delta Z_{s,t} - Z'_s \delta X_{s,t} = \int_s^t Y_u\, \dif \mathbf{X}_u - Y_s \delta X_{s,t}
		\end{align*}
		is $2\alpha$-H\"older which follows from \eqref{eqn:bound_rough_integral} and the triangle inequality. Note that calculating the bound for $\|Z,Z'\| = \|Z'\|_{\alpha} + \|R^Z\|_{2 \alpha}$ directly follows from \eqref{eqn:bound_rough_integral} .
		
	\end{proof}

	\subsection{Controlled paths as a field of Banach spaces}\label{sec:controlled_paths_fobs}
	The statements discussed in this section are simplified versions of the more general results obtained in \cite{GVRST22}. Recall that we defined a Banach space of controlled paths for every rough path $\mathbf{X}$. The question we would like to answer now is whether the indexed spaces $\{ \mathscr{D}^{\alpha}_X([0,T],W) \}_{\mathbf{X} \in \mathscr{C}^{\alpha}}$ have more structure than being just a collection of isolated spaces. This will also have practical relevance. From Theorem \ref{thm:rough_integration}, we know that rough integration induces bounded linear maps
	\begin{align*}
		\Phi(\mathbf{X},\cdot) \colon   \mathscr{D}^{\alpha}_X([0,T],W) \to \mathscr{D}^{\alpha}_X([0,T],\bar{W}).
	\end{align*}
	If $\mathbf{X}$ is a stochastic process (i.e. a random rough path), the operator norm
	\begin{align*}
		\| \Phi(\omega) \| \coloneqq \sup_{\substack{(Z,Z') \in \mathscr{D}^{\alpha}_{X(\omega)}([0,T],W) \\ (Z,Z') \neq 0}} \frac{\vertiii{\Phi(\mathbf{X}(\omega),(Z,Z'))}}{ \vertiii{Z,Z'}}
	\end{align*}
	is a natural quantity to consider (note that we dropped the lower indices for the norms on controlled rough paths spaces on the right hand side of the equation to ease notation). One seemingly basic question to answer first is the measurability of this random number. We will see that having some additional structure on the space of controlled paths will help us to answer this question.
	We make the following definition:

	\begin{definition}\label{def:cont_field_Bspaces}
		Let $\mathcal{X}$ be a topological space and $\{ E_x\}_{x \in \mathcal{X}}$ a collection of Banach spaces. $\{E_x \}_{x \in \mathcal{X}}$ is called a \emph{separable continuous field of Banach spaces} if there exists a countable set of sections $\Delta \subset \prod_{x \in \mathcal{X}}E_x$, i.e. every $g \in \Delta$ is a map $g \colon \mathcal{X} \to \bigcup_{x \in \mathcal{X}} E_x$ with $g(x) \in E_x$ for every $x \in \mathcal{X}$, that has the following properties:
		\begin{enumerate}
			\item  For every $g \in \Delta$, $x \mapsto \|g(x)\|_{E_x} \in \R$ is continuous.
			\item For every $x \in \mathcal{X}$, the set $\{g(x) \, :\, g \in \Delta \}$ is dense in $E_x$.
		\end{enumerate}
		
	\end{definition}

	\begin{remark}
		The usual definition of a continuous field of Banach spaces in the literature differs slightly from the one we gave in Definition \ref{def:cont_field_Bspaces}. In \cite{Dix77}, the definition of a continuous field of Banach spaces assumes the existence of a \emph{linear subspace} of sections $\Delta'$ satisfying (1) and (2). Separability in \cite{Dix77} means that there is a countable subset $\Delta \subset \Delta'$ satisfying (2.). It is clear that our definition is equivalent since a linear subspace of sections can be just obtained by considering the linear span of $\Delta$. Also, \cite{Dix77} assumes a third property for $\Delta'$ that is as follows:
		\begin{enumerate}
			\item[(3')] Let $\tilde{g} \in \prod_{x \in \mathcal{X}} E_x$. If for every $y \in \mathcal{X}$ and $\varepsilon > 0$, there exists $g_y \in \Delta'$ such that $\| \tilde{g}(x) - g_y(x) \|_{E_x} \leq \varepsilon$ in some neighbourhood of $y$ in $\mathcal{X}$, then $\tilde{g} \in \Delta'$.
		\end{enumerate}
		However, one can show that having a $\Delta'$ satisfying only (1) and (2), one can take some ``completion'' of $\Delta'$ that satisfies (3'), too \cite[10.2.3. Proposition]{Dix77}. Therefore, the definition we gave here could also be called a separable continuous \emph{pre-field} of Banach spaces.
	\end{remark}
	The question we want to answer now is whether the spaces of controlled paths form a separable continuous field of Banach spaces. However, we cannot expect that separability holds since the spaces of controlled paths are equipped with H\"older-type norms that make them not separable themselves. Nevertheless, we will see that a slightly weaker result holds. Inspired by the little H\"older and geometric rough paths spaces, we define:

	\begin{definition}
		Let $\mathbf{X} \in \mathscr{C}^{\beta}$, $\beta \in (1/3,1/2]$ and $\alpha \leq \beta$. We define $\mathscr{D}^{\alpha, \beta}_X([0,T],W)$ to be the closure of the space $\mathscr{D}^{\beta}_X([0,T],W)$ in the $\vertiii{\cdot}_{\alpha}$-Norm.
	\end{definition}
	The key result is the following lemma.
	
	\begin{lemma}\label{lemma:density_key}
		Let $\mathbf{X} \in \mathscr{C}^{\gamma}$ and $\frac{1}{3} < \alpha < \beta \leq \gamma \leq \frac{1}{2}$. Then the set 
		\begin{align*}
			\mathcal{Z} := \left\{ (Z,Z')\, :\, Z_t = \int_0^t \phi_u\, \dif X_u + \psi_t, \ Z'_t = \phi_t \, :\, \phi \in \mathcal{C}^{\infty}([0,T],L(\R^d, W)),\
			\psi \in \mathcal{C}^{\infty}([0,T],W) \right\} 
		\end{align*}
		is dense in $\mathscr{D}^{\alpha, \beta}_X([0,T],W)$. The integral here is defined as a Young integral.
	\end{lemma}

	\begin{proof}
		It suffices to proof that $\mathcal{Z}$ is dense in $\mathscr{D}^{\beta}_X([0,T],W)$ equipped with the norm $\vertiii{\cdot}_{\alpha}$. Let $(\xi,\xi') \in \mathscr{D}^{\beta}_X([0,T],W)$ with remainder $R^{\xi}$, i.e. $\| \xi' \|_{\beta} < \infty$ and $\|R^{\xi}\|_{2 \beta} < \infty$. Let 
		\begin{align*}
			\mathcal{P} = \{0 = t_0 < t_1 < \ldots < t_n = T\}
		\end{align*}
		be a partition with $|\mathcal{P}| = |t_{i+1} - t_i| \eqqcolon \theta > 0$ for all $i = 0,\ldots,n-1$. Define $\bar{\xi}' \colon [0,T] \to W$ to be the piecewise-linear approximation of $\xi'$ w.r.t. to $\mathcal{P}$, i.e.
		\begin{align*}
			\bar{\xi}'_t \coloneqq \xi'_{t_i} + \frac{t - t_i}{\theta} (\xi'_{t_{i+1}} - \xi'_{t_i}), \quad t \in [t_i, t_{i+1}].
		\end{align*}
		Our goal is to find a function $\psi$ with $\psi_0 = \xi_0$ such that for 
		\begin{align*}
			\bar{\xi}_t \coloneqq \int_0^t \bar{\xi}'_u\, \dif X_u + \psi_t,
		\end{align*}
		we have $\vertiii{(\xi,\xi') - (\bar{\xi},\bar{\xi}')}_{\alpha} = \|(\xi,\xi') - (\bar{\xi},\bar{\xi}')\|_{X,\alpha} \leq \varepsilon$  for any given $\varepsilon > 0$ as $\theta \to 0$.  Set $\eta := \xi' - \bar{\xi}'$. It is straightforward to show that $\|\eta\|_{\alpha} \to 0$ as $\theta \to 0$. It remains to show that
		\begin{align*}
			\|R^{\xi} - R^{\bar{\xi}} \|_{2 \alpha} \to 0
		\end{align*}
		as $\theta \to 0$ where 
		\begin{align*}
			R^{\bar{\xi}}_{s,t} = \delta \bar{\xi}_{s,t} - \bar{\xi}'_s \delta X_{s,t} = \int_s^t \bar{\xi}'_u \, \dif X_u - \bar{\xi}'_s \delta X_{s,t} + \delta \psi_{s,t}
		\end{align*}
		for a $\psi$ still to be chosen. For $s,t \in [0,T]$, we define
		\begin{align*}
			\rho_{s,t} \coloneqq \int_s^t \bar{\xi}'_u \, \dif X_u - \bar{\xi}'_s \delta X_{s,t} = \int_s^t \delta \bar{\xi}'_{s,u} \, \dif X_u.
			%= \frac{\delta \xi'_{t_i,t_{i+1}}}{\theta} \int_s^t (u-s)\, \dif X_u.
		\end{align*}
		If $s,t \in [t_i,t_{i+1}]$, we have
		\begin{align*}
			\rho_{s,t} = \frac{\delta \xi'_{t_i,t_{i+1}}}{\theta} \int_s^t (u-s)\, \dif X_u.
		\end{align*}
		Using the estimate for the Young integral in Theorem \ref{theorem:young_integral}, we see that
		\begin{align*}
			\| \rho \|_{2 \alpha;[t_i,t_{i+1}]} \leq C \|\xi' \|_{\beta} \|X\|_{\gamma} \theta^{\gamma + \beta - 2 \alpha}.
		\end{align*}
		Now we take $t_j, t_k \in \mathcal{P}$, $k < j$. Then,
		\begin{align*}
			\rho_{t_k,t_j} &= \sum_{k \leq i < j} \left[ \int_{t_i}^{t_{i+1}} \bar{\xi}'_{t_i,u}\, \dif X_u + \delta \bar{\xi}'_{t_k,t_i} \delta X_{t_i,t_{i+1}} \right] \\
			&= \sum_{k \leq i < j} \left[ \rho_{t_i,t_{i+1}} - R^{\xi}_{t_i,t_{i+1}} + \delta \xi_{t_i,t_{i+1}} - \xi_{t_k} \delta X_{t_i,t_{i+1}} \right] \\
			&= \sum_{k \leq i < j} \left[ \rho_{t_i,t_{i+1}} - R^{\xi}_{t_i,t_{i+1}} \right] + R^{\xi}_{t_k,t_j}.
		\end{align*}
		Setting $\tilde{\rho}_{s,t} \coloneqq R^{\xi}_{s,t} - \rho_{s,t}$, the calculation above implies that
		\begin{align}\label{eqn:rhotilde}
			\tilde{\rho}_{t_k,t_j} = \sum_{k \leq i < j} \left[ \rho_{t_i,t_{i+1}} - R^{\xi}_{t_i,t_{i+1}} \right].
		\end{align}
		We define $\tilde{\psi}$ to be the continuous, piecewise-linear function satisfying $\tilde{\psi}_0 = \xi_0$ and
		\begin{align*}
			\delta \tilde{\psi}_{s,t} = \frac{t-s}{t_{i+1} - t_i} (R^{\xi}_{t_i,t_{i+1}} - \rho_{t_i, t_{i+1}}),\quad s,t \in [t_i, t_{i+1}].
		\end{align*}
		With this choice, 
		\begin{align*}
			R^{\bar{\xi}}_{s,t} = \int_s^t \bar{\xi}'_u \, \dif X_u - \bar{\xi}'_s \delta X_{s,t} + \delta \tilde{\psi}_{s,t} = \rho_{s,t} +  \delta \tilde{\psi}_{s,t}.
		\end{align*}
		Now let $s,t \in \mathcal{P}$ with $t_k \leq s \leq t_{k+1} \leq \cdots \leq t_j \leq t \leq t_{j+1}$. By \eqref{eqn:rhotilde},
		\begin{align*}
			\delta \tilde{\psi}_{s,t} = \delta \tilde{\psi}_{s,t_{k+1}} + \delta \tilde{\psi}_{t_{k+1}, t_{k+2}} + \ldots +  \delta \tilde{\psi}_{t_{j}, t} = \delta \tilde{\psi}_{s,t_{k+1}} + \delta \tilde{\psi}_{t_{j}, t} + \tilde{\rho}_{t_{k+1},t_j}.
		\end{align*}
		Furthermore,
		\begin{align*}
			\rho_{s,t} = \rho_{s,t_{k+1}} + \rho_{t_{k+1}, t_j} + \rho_{t_j, t} + \delta \bar{\xi}'_{s,t_{k+1}} \delta X_{t_{k+1}, t_j} + \delta \bar{\xi}'_{t_{k+1},t_j} \delta X_{t_j, t}
		\end{align*}
		and
		\begin{align*}
			\tilde{\rho}_{s,t} = \tilde{\rho}_{s,t_{k+1}} + \tilde{\rho}_{t_{k+1}, t_j} + \tilde{\rho}_{t_j, t} + \delta \eta_{s,t_{k+1}} \delta X_{t_{k+1}, t_j} + \delta \eta_{t_{k+1},t_j} \delta X_{t_j, t}.
		\end{align*}
		Thus, we obtain that
		\begin{align*}
			R^{\bar{\xi}}_{s,t} - R^{\xi}_{s,t} &=  \delta \tilde{\psi}_{s,t} - \tilde{\rho}_{s,t} \\
			&= \delta \tilde{\psi}_{s,t_{k+1}} + \delta \tilde{\psi}_{t_j,t} - \tilde{\rho}_{s,t_{k+1}} - \tilde{\rho}_{t_j,t} - \delta \eta_{s,t_{k+1}} \delta X_{t_{k+1}, t_j} - \delta \eta_{t_{k+1},t_j} \delta X_{t_j, t}.
		\end{align*}
		Each term can now be estimated separately and we can conclude that indeed
		\begin{align*}
			\|R^{\bar{\xi}} - R^{\xi} \|_{2 \alpha} \to 0
		\end{align*}
		as $\theta \to 0$. It remains to argue that we can replace the piecewise smooth functions $\bar{\xi}'$ and $\tilde{\psi}$ by genuine smooth functions. This, however, does not cause any problems since we can approximate any continuous function arbitrarily close my smooth functions in the H\"older metric. Therefore, our claim is proved.
	\end{proof}
	Finally, the previous Lemma  yields:

	\begin{proposition}\label{prop:controlled_paths_fobs}
		Let $\frac{1}{3} < \alpha < \beta \leq \gamma \leq \frac{1}{2}$. Then the family $\{\mathscr{D}^{\alpha, \beta}_X\}_{\mathbf{X} \in \mathscr{C}^{\gamma}}$ is a separable continuous field of Banach spaces.
	\end{proposition}

	\begin{proof}
		Let $\mathcal{S}$ and $\mathcal{S}'$ be a countable dense subsets of $\mathcal{C}^{\infty}([0,T],L(\R^d, W))$ resp. $\mathcal{C}^{\infty}([0,T],W)$. Then we can define $\Delta$ as the set of maps $g \colon \mathcal{C}^{\gamma} \to \bigcup_{X \in \mathcal{C}^{\gamma}} \mathscr{D}^{\alpha, \beta}_X$ given by $g(X) = (Z,Z')$ where
		\begin{align*}
			Z_t = \int_0^t \phi_u\, \dif X_u + \psi_t, \ Z'_t = \phi_t
		\end{align*}
		with $\phi \in \mathcal{S}$ and $\psi \in \mathcal{S}'$. The claimed properties now follow from continuity of the Young integral, cf. Theorem \ref{theorem:young_integral}, and Lemma \ref{lemma:density_key}.
		
	\end{proof}
	Remember that we considered the measurability question of the operator norm of a family of linear mappings
	\begin{align*}
		\Phi(\mathbf{X}(\omega),\cdot) \colon   \mathscr{D}^{\alpha}_{X(\omega)}([0,T],W) \to \mathscr{D}^{\alpha}_{X(\omega)}([0,T],\bar{W})
	\end{align*}
	(like rough integration, for instance). We will formulate a corresponding result now.

	\begin{proposition}\label{prop:lin_maps_meas}
		Let $\frac{1}{3} < \alpha < \beta \leq \gamma \leq \frac{1}{2}$ and let $\Delta$ be the set of sections given in the definition of a continuous field of Banach spaces. Assume that for every rough path $\mathbf{X} \in \mathscr{C}^{\gamma}$, there is a bounded linear map 
		\begin{align*}
			\Phi(\mathbf{X},\cdot) \colon  \mathscr{D}^{\alpha, \beta}_{X}([0,T],W) \to  \mathscr{D}^{\alpha, \beta}_{X}([0,T],\bar{W})
		\end{align*}
		that satisfies the property that $\mathbf{X} \mapsto \vertiii{\Phi(\mathbf{X}, g(X))}$ is continuous for every $g \in \Delta$. Let $\mathbf{X}(\omega)$ be a random rough path with the property that $\omega \mapsto \mathbf{X}(\omega)$ is measurable. Then the operator norm
		\begin{align*}
			\| \Phi(\omega) \| = \sup_{\substack{(Z,Z') \in \mathscr{D}^{\alpha, \beta}_{X(\omega)}([0,T],W) \\ (Z,Z') \neq 0}} \frac{\vertiii{\Phi(\mathbf{X}(\omega),(Z,Z'))}}{\vertiii{Z,Z'}}
		\end{align*}
		is measurable.
	\end{proposition}

	\begin{proof}
		For every $\omega \in \Omega$,
		\begin{align*}
			\| \Phi(\omega) \| = \sup_{\substack{(Z,Z') \in \mathscr{D}^{\alpha, \beta}_{X(\omega)}([0,T],W) \\ (Z,Z') \neq 0}} \frac{\vertiii{\Phi(\mathbf{X}(\omega),(Z,Z'))}}{\vertiii{Z,Z'}} = \sup_{g \in \Delta} \frac{\vertiii{\Phi(\mathbf{X}(\omega),g(X(\omega)))}}{\vertiii{g(X(\omega))}}.
		\end{align*}
		By our assumptions, $\omega \mapsto \frac{\vertiii{ \Phi(\mathbf{X}(\omega),g(X(\omega))) }}{\vertiii{g(X(\omega))}}$ is measurable for every fixed $g \in \Delta$. Since $\Delta$ is countable, the result follows.
		
	\end{proof}

	To apply Proposition \ref{prop:lin_maps_meas} to the rough integration map, we still have to prove that
	\begin{align*}
		\mathbf{X} \to \left\| \int g(X) \, \dif \mathbf{X} \right\|_{\mathscr{D}^{\alpha,\beta}_X}
	\end{align*}
	is continuous. This will follow by a more general result on rough integration, cf. the forthcoming Theorem \ref{thm:stability_rough_integral} and Corollary \ref{cor:cont_sect_int}. \smallskip

	It is known that every continuous field of Banach spaces $\{E_x\}_{x \in \mathcal{X}}$ induces a natural topology on the total space $E \coloneqq \bigsqcup_{x \in \mathcal{X}} E_x$. To describe it, we introduce the projection $p \colon E \to \mathcal{X}$, i.e. if $Z \in E_x$, $p(Z) = x$. We define for $g \in \Delta$, an open set $U \subset \mathcal{X}$ and $\varepsilon > 0$ the \emph{tube}
	\begin{align*}
		W(g,U,\varepsilon) \coloneqq \{Z \in E \, :\, p(Z) \in U,\ \|Z - g(p(Z)) \|_{E_{p(Z)}} < \varepsilon \},
	\end{align*}
	see the picture below.
	\begin{figure}[H]
		\includegraphics[width=10cm]{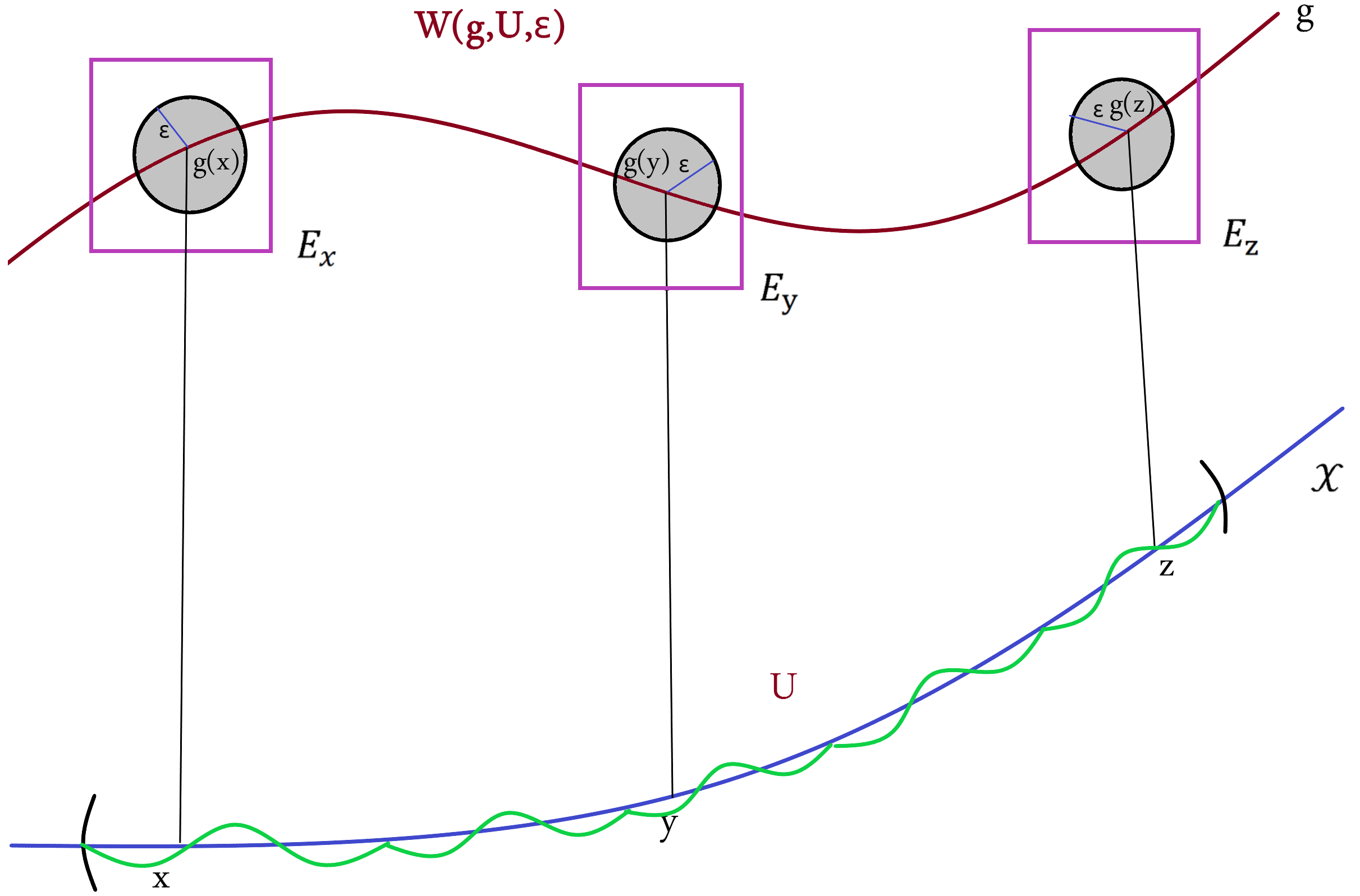}
		\caption{Open tube}
	\end{figure}
	
	The topology defined on $E$ is the smallest one containing the tubes as open sets. It is also called \emph{tube topology}.
	\smallskip
	Fortunately, in the case of controlled paths, the tube topology is completely metrizable with an explicit metric. We state this result now.

	\begin{proposition}
		Let $\alpha < \beta \leq \frac{1}{2}$ and $\mathscr{D} \coloneqq \bigsqcup_{\mathbf{X} \in \mathscr{C}^{\beta}} \mathscr{D}^{\alpha, \beta}_{X}([0,T],W)$. Then the tube topology on $\mathscr{D}$ is completely metrizable with metric given by
		\begin{align*}
			d^{\flat}_{\alpha,\beta}((Y,Y'),(\tilde{Y},\tilde{Y}')) &\coloneqq  \varrho_{\beta}(p(Y,Y'),p(\tilde{Y},\tilde{Y}')) + \|Y' - \tilde{Y}'\|_{\alpha} + \|R^Y - R^{\tilde{Y}} \|_{2 \alpha} \\
			&\quad + |Y_0 - \tilde{Y}_0| + |Y'_0 - \tilde{Y}'_0|.
		\end{align*}
		If we replace $\mathscr{C}^{\beta}$ by $\mathscr{C}^{\beta}_g$, $\mathscr{D}$ is also separable, i.e. Polish.			
	\end{proposition}

	\begin{proof}
		We fix some notation first. For given $(Y,Y') \in \mathscr{D}^{\alpha, \beta}_{X}([0,T],W)$ and $(\tilde{Y},\tilde{Y}') \in  \mathscr{D}^{\alpha, \beta}_{\tilde{X}}([0,T],W)$, we set 
		\begin{align*}
			\vertiii{(Y,Y') ; (\tilde{Y},\tilde{Y}')}_{\alpha} \coloneqq \|Y' - \tilde{Y}'\|_{\alpha} + \|R^Y - R^{\tilde{Y}} \|_{2 \alpha}  + |Y_0 - \tilde{Y}_0| + |Y'_0 - \tilde{Y}'_0|.
		\end{align*}
		For given $(Y,Y') \in \mathscr{D}$ and $\varepsilon > 0$, we define
		\begin{align*}
			B_{\varepsilon} (Y,Y') \coloneqq \{(\tilde{Y},\tilde{Y}') \in \mathscr{D} \, :\, d^{\flat}_{\alpha,\beta}((Y,Y'),(\tilde{Y},\tilde{Y}')) < \varepsilon \}.
		\end{align*}
		For $\mathbf{X} \in \mathscr{C}^{\beta}$ and $\eta > 0$, we use the notation
		\begin{align*}
			B_{\eta}(\mathbf{X}) \coloneqq \{\tilde{\mathbf{X}} \in \mathscr{C}^{\beta} \,:\, \varrho(\mathbf{X},\tilde{\mathbf{X}}) < \eta \}.
		\end{align*}
		Recall the definition of $\Delta$ given in the proof of Proposition \ref{prop:controlled_paths_fobs}. \smallskip
		
		\textbf{Claim 1:} For given $(Y,Y') \in \mathscr{D}$ and $\varepsilon > 0$, there is an open set $U \subset \mathscr{C}^{\beta}$, an element $g \in \Delta$ and a number $\delta > 0$ such that
		\begin{align*}
			(Y,Y') \in W(g,U,\delta) \subseteq  B_{\varepsilon} (Y,Y').
		\end{align*}
		
		To prove this claim, for $\mathbf{X} = p(Y,Y')$, we define $ U \coloneqq B_{\eta}(\mathbf{X})$  where $\eta > 0$ will be chosen later. For given $\delta > 0$, we choose $g = g_{\delta} \in \Delta$ such that
		\begin{align*}
			\vertiii{(Y,Y') - g(p(Y,Y'))}_{X,\alpha} < \delta. 
		\end{align*}
		With these choices, we always have that $(Y,Y') \in W(g,U,\delta)$. Now let $(\tilde{Y},\tilde{Y}') \in   W(g,U,\delta)$ be arbitrary and set  $\tilde{\mathbf{X}} = p(\tilde{Y},\tilde{Y}')$. Note that
		\begin{align*}
			&d^{\flat}_{\alpha,\beta}((Y,Y'),(\tilde{Y},\tilde{Y}')) \\
			=\ &\varrho_{\beta}(\mathbf{X},\tilde{\mathbf{X}}) + \vertiii{(Y,Y') ; (\tilde{Y},\tilde{Y}')}_{\alpha} \\
			<\ &\eta +   \vertiii{(Y,Y') - g(p(Y,Y'))}_{X,\alpha} +  \vertiii{(\tilde{Y},\tilde{Y}') - g(p(\tilde{Y},\tilde{Y}'))}_{\tilde{X},\alpha} + \vertiii{g(p(Y,Y')) ; g(p(\tilde{Y},\tilde{Y}'))}_{\alpha} \\
			<\ &\eta + 2\delta + \vertiii{g(p(Y,Y')) ; g(p(\tilde{Y},\tilde{Y}'))}_{\alpha}.
		\end{align*}e
		Using continuity of the Young integral, we can deduce the bound
		\begin{align}\label{eqn:cont_young_fields}
			\vertiii{g(p(Y,Y')) ; g(p(\tilde{Y},\tilde{Y}'))}_{\alpha} \leq C_{g} \varrho_{\beta}(\mathbf{X},\tilde{\mathbf{X}}) \leq C_g \eta.
		\end{align}
		Therefore, if $\varepsilon > 0$ is given, we first choose $0 < \delta < \varepsilon/4$ and then $\eta > 0$ such that $\eta(1 + C_g) < \varepsilon/2$ to obtain that $d^{\flat}_{\alpha,\beta}((Y,Y'),(\tilde{Y},\tilde{Y}')) < \varepsilon$. This proves claim 1. \smallskip
		
		\textbf{Claim 2:} For given $W(g,U,\delta)$ and $(Y,Y') \in W(g,U,\delta)$, there is an $\varepsilon > 0$ such that
		\begin{align*}
			B_{\varepsilon} (Y,Y') \subseteq W(g,U,\delta).
		\end{align*}
		
		To see this, let $ \mathbf{X} = p(Y,Y')$. By definition, $\mathbf{X} \in U$ and since $U$ is open, there is an $\eta > 0$ such that $B_{\eta}(\mathbf{X}) \subseteq U$. Let $(\tilde{Y},\tilde{Y}') \in B_{\varepsilon} (Y,Y')$ be arbitrary and $\tilde{\mathbf{X}} = p(\tilde{Y},\tilde{Y}')$. If $0 < \varepsilon < \eta$, it follows that
		\begin{align*}
			\tilde{\mathbf{X}} \in B_{\eta}(\mathbf{X}) \subseteq U.
		\end{align*}
		In remains to show that choosing $\varepsilon > 0$ sufficiently small, we can obtain that
		\begin{align*}
			\vertiii{(\tilde{Y},\tilde{Y}') - g(p(\tilde{Y},\tilde{Y}'))}_{\tilde{X},\alpha} < \delta.
		\end{align*}
		Note that
		\begin{align*}
			&\left| \vertiii{(\tilde{Y},\tilde{Y}') - g(p(\tilde{Y},\tilde{Y}'))}_{\tilde{X},\alpha} - \vertiii{(Y,Y') - g(p(Y,Y'))}_{X,\alpha} \right| \\
			\leq\  &\vertiii{(Y,Y') ; (\tilde{Y},\tilde{Y}')}_{\alpha} +  \vertiii{g(p(Y,Y')) ; g(p(\tilde{Y},\tilde{Y}'))}_{\alpha}.
		\end{align*}
		Using again \eqref{eqn:cont_young_fields} and the assumption, the right hand side gets small when $\varepsilon$ is chosen small. Therefore, for any given $\nu > 0$, we can choose $\varepsilon > 0$ sufficiently small to obtain
		\begin{align*}
			\vertiii{(\tilde{Y},\tilde{Y}') - g(p(\tilde{Y},\tilde{Y}'))}_{\tilde{X},\alpha} \leq \nu + \vertiii{(Y,Y') - g(p(Y,Y'))}_{X,\alpha}.
		\end{align*}
		Since $ \vertiii{(Y,Y') - g(p(Y,Y'))}_{X,\alpha} < \delta$, we can find a $\nu > 0$ such that
		\begin{align*}
			\nu + \vertiii{(Y,Y') - g(p(Y,Y'))}_{X,\alpha} < \delta.
		\end{align*}
		From these observations, we can deduce the second claim. Both claims together prove that $d^{\flat}_{\alpha,\beta}$ indeed metrizes the tube topology. The fact that $\mathscr{D}$ is complete with respect to $d^{\flat}_{\alpha,\beta}$ follows from completeness of the space $\mathscr{C}^{\beta}$ with respect to $\varrho_{\beta}$ and completeness of the spaces $\mathscr{D}_X^{\alpha,\beta}$. Separability follows from separability of the respective spaces.

	\end{proof}
	
	We can now prove an important stability result for rough integration.

	\begin{theorem}\label{thm:stability_rough_integral}
		Let $\mathbf{X}, \tilde{\mathbf{X}} \in \mathscr{C}^{\beta}$,  $(Y,Y') \in \mathscr{D}^{\alpha, \beta}_{X}$ and $(\tilde{Y},\tilde{Y}') \in \mathscr{D}^{\alpha, \beta}_{\tilde{X}}$. Set 
		\begin{align*}
			Z \coloneqq \int_0^{\cdot} Y_u \, \dif \mathbf{X}_u, \quad Z' \coloneqq Y
		\end{align*}
		and define $(\tilde{Z},\tilde{Z}')$ similarly. Then, locally,
		\begin{align*}
			d^{\flat}_{\alpha,\beta}((Z,Z'),(\tilde{Z},\tilde{Z}')) \leq C\, d^{\flat}_{\alpha,\beta}((Y,Y'),(\tilde{Y},\tilde{Y}')).
		\end{align*}
		In other words: the integration map
		\begin{align*}
			(Y,Y') \mapsto \left( \int Y \, \dif p(Y,Y'), Y \right)
		\end{align*}
		is locally Lipschitz continuous.
	\end{theorem}

	\begin{proof}
		It suffices to establish a bound for $\|R^Z - R^{\tilde{Z}}\|_{2 \alpha}$. Recall that
		\begin{align*}
			R^Z_{s,t} = \int_s^t Y_u \, \dif \mathbf{X}_u - Y_s \delta X_{s,t} = (\mathcal{I} \Xi)_{s,t} - \Xi_{s,t} + Y'_s \mathbb{X}_{s,t}
		\end{align*}
		where $\Xi_{u,v} = Y_u \delta X_{u,v} + Y'_u \mathbb{X}_{u,v}$ and $\mathcal{I}$ is the integration map provided by the Sewing lemma. A similar decomposition holds for $R^{\tilde{Z}}_{s,t}$ with $\Xi$ replaced by $\tilde{\Xi}_{u,v} =  \tilde{Y}_u \delta \tilde{X}_{u,v} + \tilde{Y}'_u \tilde{\mathbb{X}}_{u,v}$. Setting $\Psi \coloneqq \Xi - \tilde{\Xi}$, linearity of $\mathcal{I}$ yields
		\begin{align*}
			|R^Z_{s,t} - R^{\tilde{Z}}_{s,t}| \leq |(\mathcal{I} \Psi)_{s,t} - \Psi_{s,t}| + |Y'_s \mathbb{X}_{s,t} - \tilde{Y}'_s \tilde{\mathbb{X}}_{s,t}|.
		\end{align*}
		The Sewing lemma gives us the bound
		\begin{align*}
			|(\mathcal{I} \Psi)_{s,t} - \Psi_{s,t}| \leq C \| \delta \Psi \|_{3 \alpha} |t-s|^{3\alpha}.
		\end{align*}
		We have
		\begin{align*}
			\delta \Psi_{s,u,t} = \delta \Xi_{s,u,t} - \delta \tilde{\Xi}_{s,u,t} = R^{\tilde{Y}}_{s,u} \tilde{X}_{u,t} + \delta \tilde{Y}'_{s,u} \tilde{\mathbb{X}}_{u,t} - R^Y_{s,u} X_{u,t} - \delta Y'_{s,u} \mathbb{X}_{u,t}.
		\end{align*}
		Therefore, by using the triangle inequality:
		\begin{align*}
			\| \delta \Psi \|_{3 \alpha} \leq C d^{\flat}_{\alpha,\beta}((Y,Y'),(\tilde{Y},\tilde{Y}')).
		\end{align*}
		The triangle inequality also yields 
		\begin{align*}
			|Y'_s \mathbb{X}_{s,t} - \tilde{Y}'_s \tilde{\mathbb{X}}_{s,t}| \leq C |t-s|^{2\alpha} d^{\flat}_{\alpha,\beta}((Y,Y'),(\tilde{Y},\tilde{Y}'))
		\end{align*}
		which concludes the proof.
	\end{proof}

	\begin{corollary}\label{cor:cont_sect_int}
		For every $g \in \Delta$, the map 
		\begin{align*}
			\mathbf{X} \to \left\| \int g(X) \, \dif \mathbf{X} \right\|_{\mathscr{D}^{\alpha,\beta}_X}
		\end{align*}
		is continuous.
	\end{corollary}
	\begin{proof}
		For $\mathbf{X}, \tilde{\mathbf{X}} \in \mathscr{C}^{\beta}$, the reverse triangle inequality for H\"older norms gives
		\begin{align*}
			\left| \left\| \int g(X) \, \dif \mathbf{X} \right\|_{\mathscr{D}^{\alpha,\beta}_X} - \left\| \int g(\tilde{X}) \, \dif \tilde{\mathbf{X}} \right\|_{\mathscr{D}^{\alpha,\beta}_{\tilde{X}}} \right| &\leq d^{\flat}_{\alpha,\beta} \left(  \int g(X) \, \dif \mathbf{X} ,  \int g(\tilde{X}) \, \dif \tilde{\mathbf{X}} \right) \\
			&\leq\ C\, d^{\flat}_{\alpha,\beta}(g(X), g(\tilde{X}))
		\end{align*}
		locally.  Recall that 
		\begin{align*}
			g(X) = \left( \int \phi \, \dif X + \psi , \phi \right)
		\end{align*}
		for some smooth functions $\phi$ and $\psi$. Therefore, we can use continuity of the Young integral to see that
		\begin{align*}
			d^{\flat}_{\alpha,\beta}(g(X), g(\tilde{X})) \leq C \varrho_{\beta}(\mathbf{X},\tilde{\mathbf{X}})
		\end{align*}
		locally and continuity follows.
	\end{proof}

	\section{Rough differential equations}
	
	Having defined the rough integral, we can now say how a general \emph{rough differential equation} should be understood.
	
	\begin{definition}
		Let $\mathbf{X} \in \mathscr{C}^{\alpha}$, $\frac{1}{3} < \alpha \leq \frac{1}{2}$, $\sigma = (\sigma_1,\ldots,\sigma_d)$ a collection of vector fields $\sigma_i \colon \R^m \to \R^m$ and $y \in \R^m$. We call $Y \colon [0,T] \to \R^m$ a \emph{solution to the rough differential equation} (RDE)
		\begin{align*}
			\dif Y_t &= \sigma(Y_t)\, \dif \mathbf{X}_t; \quad t \in [0,T], \\
			Y_0 &= y,
		\end{align*}
		if and only if $t \mapsto \sigma(Y_t)$ is controlled by $\mathbf{X}$ and satisfies the integral equation
		\begin{align}\label{eqn:rough_integral_eq}
			Y_t = y + \int_0^t \sigma(Y_s)\, \dif \mathbf{X}_s 
		\end{align}
		where the integral is understood as a rough integral.
	\end{definition}
	
	Since rough integrals are also controlled paths, any solution $Y$ that satisfies \eqref{eqn:rough_integral_eq} will be controlled by $\mathbf{X}$, too. A natural candidate for a Gubinelli derivative of $Y$ is $\sigma(Y)$. We would therefore like to consider the map 
	\begin{align*}
		\mathcal{M}(Y,Y') := \left( y + \int_0^{\cdot} \sigma(Y_s)\, \dif \mathbf{X}_s, \sigma(Y) \right)
	\end{align*}
	as a map from the space of controlled paths to itself and try show that that it is a contraction on a small time interval. To properly define this map, one has to show that the composition of a controlled path with a sufficiently smooth function $\sigma$ is again controlled.
	
	\begin{lemma}\label{lemma:comp_contr_smooth}
		Let $\mathbf{X} \in \mathscr{C}^{\alpha}$, $(Y,Y') \in \mathscr{D}^{\alpha}_X([0,T],W)$ and let $\varphi \colon W \to \bar{W}$ be twice continuously differentiable. Then the path $t \mapsto \varphi(Y_t)$ is again controlled by $X$ with a Gubinelli derivative given by $\varphi(Y)'_t = D \varphi(Y_t)Y'_t$. Moreover, if $\varphi$ is bounded with bounded derivatives, the estimate
		\begin{align*}
			\|\varphi(Y),\varphi(Y)' \|_{X,\alpha} \leq C(\|Y\|_{\alpha} + \|Y\|^2_{\alpha} + \|Y,Y'\|_{X,\alpha})
		\end{align*}
		holds where $C$ depends on $\| \varphi \|_{\mathcal{C}^2}$.
	\end{lemma}

	\begin{proof}
		It suffices to consider the case of $\sigma$ being bounded with bounded derivatives, the general case follows by localization. We have
		\begin{align*}
			\| \varphi(Y) \|_{\alpha} \leq \|D \sigma \|_{\infty} \|Y\|_{\alpha}
		\end{align*}
		and
		\begin{align*}
			\| \varphi(Y)' \|_{\alpha} &= \|D \varphi(Y)Y' \|_{\alpha} \leq \|D \varphi(Y)\|_{\alpha} \|Y'\|_{\infty} + \|D \varphi(Y)\|_{\infty} \|Y'\|_{\alpha} \\
			&\leq \|D^2 \sigma \|_{\infty} \|Y\|_{\alpha} + \|D \sigma\|_{\infty} \|Y'\|_{\alpha}.
		\end{align*}
		This shows that $\varphi(Y), \varphi(Y)' \in \mathcal{C}^{\alpha}$. We have to prove that
		\begin{align*}
			R^{\varphi}_{s,t} &\coloneqq R^{\varphi(Y)}_{s,t} \coloneqq \delta \varphi(Y)_{s,t} - \sigma(Y)'_s \delta X_{s,t} \\
			&= \delta \varphi(Y)_{s,t} - D \varphi(Y_s) Y'_s \delta X_{s,t}
		\end{align*}
		is $2\alpha$-H\"older. Since
		\begin{align*}
			R^{\varphi}_{s,t} = \varphi(Y_t) - \varphi(Y_s) - D \varphi(Y_s) \delta Y_{s,t} + D \varphi(Y_s) R^Y_{s,t},
		\end{align*}
		Taylor's theorem yields the bound
		\begin{align*}
			\|R^{\varphi} \|_{2 \alpha} \leq \frac{1}{2} \|D^2 \varphi\|_{\infty} \|Y\|_{\alpha}^2 + \|D \varphi \|_{\infty} \|R^Y\|_{2\alpha},
		\end{align*}
		which shows that indeed $(\varphi(Y),\varphi(Y)')$ is controlled by $X$ and the desired bound.
	\end{proof}
	Next, we formulate the main theorem about the non-linear rough differential equations.

	\begin{theorem}
		Let $\mathbf{X} \in \mathscr{C}^{\alpha}([0,T],\R^d)$ for $\frac{1}{3} < \alpha \leq \frac{1}{2}$, $y \in \R^m$ and $\sigma \in \mathcal{C}^3(\R^m, L(\R^d,\R^m))$. Then there exists a unique controlled path $(Y,Y') \in \mathscr{D}^{\alpha}_X([0,T],\R^m)$ with $Y' = \sigma(Y)$ that satisfies
		\begin{align*}
			Y_t = y + \int_0^t \sigma(Y_s)\, \dif \mathbf{X}_s; \quad t \in [0,T].
		\end{align*}
		
	\end{theorem}

	\begin{proof}
		The proof is very similar to the one we gave in Theorem \ref{thm:young_ode}, i.e. we will show that a properly defined mapping has a fixed point. For $(Y,Y') \in \mathscr{D}^{\alpha}_X$ and $0 < T_0 \leq T$, set
		\begin{align*}
			(Z_t,Z'_t) \coloneqq (\sigma(Y_t), D\sigma(Y_t) Y'_t) \in \mathscr{D}^{\alpha}_X; \quad t \in [0,T_0].
		\end{align*}
		We define the map 
		\begin{align*}
			\mathcal{M}(Y,Y') := \left( y + \int_0^{t} Z_s\, \dif \mathbf{X}_s, Z_t ;\ t \in [0,T_0] \right).
		\end{align*}
		The expected solution will be a fixed point of this map. We will not define this map on the whole space of controlled paths but on the closed unit ball
		\begin{align*}
			\mathcal{B}_{T_0} \coloneqq \left\{ (Y,Y') \in \mathscr{D}^{\alpha}_X\, :\, Y_0 = y, Y'_0 = \sigma(y), \|Y,Y'\|_{X,\alpha} \leq 1 \right\}
		\end{align*}
		of controlled paths starting in $(y,\sigma(y))$. We will have to prove two things:
		\begin{enumerate}
			\item  $\mathcal{M}$ leaves $\mathcal{B}_{T_0}$ invariant, i.e. $\mathcal{M} \colon \mathcal{B}_{T_0} \to \mathcal{B}_{T_0}$ is a well defined map,
			\item  $\mathcal{M}$ is a contraction.
		\end{enumerate}
		We start with the first point. Clearly, $\mathcal{M}(Y,Y')_0 = (y,\sigma(y))$. To prove that $\|Y,Y'\|_{X,\alpha} \leq 1$, we use the estimate for the rough integral given in Theorem \ref{eqn:bound_rough_integral}: 
		\begin{align*}
			\|\mathcal{M} \|_{X,\alpha}  &= \| \int_0^{\cdot} Z_s\, \dif \mathbf{X}_s, Z \|_{X,\alpha} \\
			&\leq \|Z\|_{\alpha} + \|Z'\|_{\alpha} \|\mathbb{X}\|_{2 \alpha} + CT_0^{\alpha}(\| X \|_{\alpha} \|R^Z\|_{2 
				\alpha} + \| \mathbb{X} \|_{2 \alpha} \|Z'\|_{\alpha}) \\
			&\leq  \|Z\|_{\alpha} + \|Z,Z'\|_{X,\alpha} \|\mathbb{X}\|_{2 \alpha} + CT^{\alpha} \vertiii{\mathbf{X}}_{\alpha} \|Z,Z'\|_{X,\alpha}.
		\end{align*}
		We have $\|Z\|_{\alpha} \leq C\|Y\|_{\alpha}$ and
		\begin{align*}
			\|Y\|_{\alpha} &\leq  \|Y'\|_{\infty} \|X\|_{\alpha} + T_0^{\alpha} \| R^Y \|_{2\alpha} \\
			&\leq |Y'_0| \|X\|_{\alpha} + T_0^{\alpha} \|Y'\|_{\alpha} \|X\|_{\alpha} + T^{\alpha} \| R^Y \|_{2\alpha} \\
			&\leq C \|X\|_{\alpha} + T^{\alpha}_0(1 + \|X\|_{\alpha}) \|Y,Y'\|_{X,\alpha} \\
			&\leq C \|X\|_{\alpha} + T^{\alpha}_0(1 + \|X\|_{\alpha}).
		\end{align*}
		To estimate $\|Z,Z'\|_{X,\alpha}$, we use Lemma \ref{lemma:comp_contr_smooth}:
		\begin{align*}
			\|Z,Z'\|_{X,\alpha} &\leq C(\|Y\|_{\alpha} + \|Y\|^2_{\alpha} + \|Y,Y'\|_{X,\alpha}) \\
			&\leq C(1 + \|Y\|_{\alpha} + \|Y\|^2_{\alpha}).
		\end{align*}
		Note that we already estimated $\|Y\|_{\alpha}$ above. To summarize, we see that $\|\mathcal{M} \|_{X,\alpha}$ gets small if $T_0$ and $\vertiii{\mathbf{X}}_{\alpha}$ are getting small. As in the proof of Theorem \ref{thm:young_ode}, we will therefore assume first that $\mathbf{X}$ is smoother than only being $\alpha$-H\"older continuous to assure that $\vertiii{\mathbf{X}}_{\alpha}$ gets small as $T_0 \to 0$. In total, we can thus guarantee that $\mathcal{M}$ leaves $\mathcal{B}_{T_0}$ invariant for a sufficiently small $T_0 > 0$. It remains to prove that $\mathcal{M}$ is a contraction on $\mathcal{B}_{T_0}$. To do this, we have to estimate the difference between two rough integrals in the $\vertiii{\cdot}_{X,\alpha}$-norm. Note that we do not have to use Theorem \ref{thm:stability_rough_integral} since the driving rough path $\mathbf{X}$ is fixed. The complete proof for the contraction property is a bit long, but does not provide many new insights, that is why we will not present it here. It can be found in \cite[Theorem 8.3.]{FH20}. 
	\end{proof}
	
	%  Let $(Y,Y') \in \mathscr{D}^{\alpha}_{\beta}$ be the solution of a rough differential equation (RDE)
	% \begin{align*}
		% 	\dif Y_t &= \sigma(Y_t)\, \dif \mathbf{X}_t; \quad t \in [0,T], \\
		% 	Y_0 &= y,
		% \end{align*}
	% with $Y'_t = \sigma(Y_t)$. Using estimates for the rough integral, we can also find a bound for this solution.
	% 
	% 	\begin{proposition}[A priori estimate on RDE solutions]
		% 		It holds that
		% 		\begin{align*}
			% 			\|Y\|_{\alpha} \leq C \left( (\|\sigma \|_{\mathcal{C}^2} \vertiii{\mathbf{X}}_{\alpha} ) \vee (\|\sigma \|_{\mathcal{C}^2} \vertiii{\mathbf{X}}_{\alpha} )^{\frac{1}{\alpha}} \right).
			% 		\end{align*}
		% 		
		% \end{proposition}
	% \begin{proof}
		% 	\cite[Proposition 8.2]{FH20}.	
		% \end{proof}
	There is also a stability result for solutions to rough differential equations that we want to cite here. To formulate it, we define the metric
	\begin{align*}
		d^{\flat}_{\alpha}((Y,Y'),(\tilde{Y},\tilde{Y}')) &\coloneqq d^{\flat}_{\alpha,\alpha}((Y,Y'),(\tilde{Y},\tilde{Y}')) \\
		&\coloneqq \varrho_{\alpha}(\mathbf{X},\tilde{\mathbf{X}}) + \|Y' - \tilde{Y}'\|_{\alpha}  + \|R^Y - R^{\tilde{Y}} \|_{2 \alpha} \\
		&\quad + |Y_0 - \tilde{Y}_0| + |Y'_0 - \tilde{Y}'_0|
	\end{align*}
	for $(Y,Y') \in \mathscr{D}_X^{\alpha}$ and $(\tilde{Y},\tilde{Y}') \in \mathscr{D}_{\tilde{X}}^{\alpha}$ which is a metric on the total space $\mathscr{D} = \sqcup_{\mathbf{X} \in \mathscr{C}^{\alpha}} \mathscr{D}_X^{\alpha}$.
	
	\begin{theorem}[Stability of RDE solutions]\label{thm:stability_RDEs}
		Let $(Y,Y')$ and $(\tilde{Y},\tilde{Y}')$ be solutions to 
		\begin{align*}
			\dif Y_t = \sigma(Y_t)\, \dif \mathbf{X}_t; \ Y_0 = y \quad \text{resp.} \quad \dif \tilde{Y}_t = \sigma(\tilde{Y}_t)\, \dif \tilde{\mathbf{X}}_t; \ \tilde{Y}_0 = \tilde{y}
		\end{align*}
		with $Y' = \sigma(Y)$ and $\tilde{Y}' = \sigma(\tilde{Y})$. Then
		\begin{align*}
			d^{\flat}_{\alpha}((Y,Y'),(\tilde{Y},\tilde{Y}')) \leq C(|y - \tilde{y}| + \varrho_{\alpha}(\mathbf{X},\tilde{\mathbf{X}}))
		\end{align*}
		locally.
	\end{theorem}
	\begin{proof}
		\cite[Theorem 8.5]{FH20}.	
	\end{proof}

	\subsection{Rough differential equations driven by a Brownian motion} In this part, we discuss how to employ rough theory in stochastic analysis. Let us start with the following proposition which, loosely speaking, claims that It\=o (resp. Stratonovich) integration coincides with rough integration against the enhanced It\=o (resp. Stratonovich) Brownian motion.
	\begin{proposition}
		Let $B = (B^1,\ldots,B^d)$ be a $d$-dimensional Brownian motion and $\mathbf{B}^{\text{It\=o}}$ resp. $\mathbf{B}^{\text{Strat}}$ its It\=o resp. Stratonovich lift to a rough paths valued process. For $\frac{1}{3} < \alpha < \frac{1}{2}$, assume that $(Y(\omega),Y'(\omega)) \in \mathscr{D}^{\alpha}_{X(\omega)}$ almost surely and that $(Y,Y')$ is adapted to the filtration generated by $B$. Then
		\begin{align*}
			\int_0^T Y_s\, \dif B_s = \int_0^T Y_s \, \dif \mathbf{B}^{\text{It\=o}}_s \quad \text{and} \quad \int_0^T Y_s\, \circ \dif B_s = \int_0^T Y_s \, \dif \mathbf{B}^{\text{Strat}}_s
		\end{align*}
		almost surely.
	\end{proposition}

	\begin{proof}
		We will only prove the It\=o-case, the identity for  Stratonovich integral can be found in \cite[Corollary 5.2]{FH20}. It is known that
		\begin{align*}
			\int_0^T Y_s \, \dif B_s  = \lim_{|\mathcal{P}| \to 0} \sum_{[u,v] \in \mathcal{P}} Y_u \delta B_{u,v}
		\end{align*}
		in probability. Passing to a subsequence, we may assume that there is a sequence of partitions such that the convergence holds almost surely. It suffices to prove that
		\begin{align*}
			\lim_{|\mathcal{P}| \to 0} \sum_{[u,v] \in \mathcal{P}} Y'_u \mathbb{B}^{\text{It\=o}}_{u,v} = 0
		\end{align*}
		in $L^2(\Omega)$. We will assume that $\|Y'(\omega)\| \leq M$ almost surely, the general case follows by a stopping argument. Fix a partition $\mathcal{P} = \{0 = \tau_0 < \ldots \tau_N = T\}$. One can check that $(S_k)$ with $S_0 = 0$ and $S_{k+1} - S_k = Y'_{\tau_k} \mathbb{B}^{\text{It\=o}}_{\tau_{k+1},\tau_k}$ is a discrete martingale. Since its increments are uncorrelated,
		\begin{align*}
			\left\|\sum_{[u,v] \in \mathcal{P}} Y'_u \mathbb{B}^{\text{It\=o}}_{u,v} \right\|_{L^2}^2 = \sum_{[u,v] \in \mathcal{P}} \left\|Y'_u \mathbb{B}^{\text{It\=o}}_{u,v} \right\|_{L^2}^2 \leq M \sum_{[u,v] \in \mathcal{P}} \left\|\mathbb{B}^{\text{It\=o}}_{u,v} \right\|_{L^2}^2 = \mathcal{O}(|\mathcal{P}|)
		\end{align*}
		and the claim follows.
	\end{proof}

	\begin{corollary}\label{cor:Ito_Strat_RP}
		For $\sigma \in \mathcal{C}^3(\mathbb{R}^{m},L(\mathbb{R}^d,\mathbb{R}^m))$, the solutions to
		\begin{align*}
			\dif Y_t = \sigma(Y_t)\, \dif B_t \quad \text{and} \quad \dif Y_t = \sigma(Y_t)\, \dif \mathbf{B_t}^{\text{It\=o}}
		\end{align*}
		resp.
		\begin{align*}
			\dif Y_t = \sigma(Y_t)\, \circ \dif B_t \quad \text{and} \quad \dif Y_t = \sigma(Y_t)\, \dif \mathbf{B_t}^{\text{Strat}}
		\end{align*}
		with the same initial conditions coincide almost surely.
		
	\end{corollary}
	We can now prove an important theorem in stochastic analysis, the \emph{Wong-Zakai theorem}, that connects stochastic differential equations to random ordinary differential equations. 
	
	\begin{theorem}\label{thm:wong_zakai}
		Let $\sigma \in \mathcal{C}^3(\mathbb{R}^{m},L(\mathbb{R}^d,\mathbb{R}^m))$, $y \in \R^m$, $B = (B^1,\ldots,B^d)$ be a Brownian motion defined on $[0,1]$ and $B(n)$ and be its piecewise-linear approximation at the the dyadic points $0 < 2^{-n} < \ldots < (2^n - 1)2^{-n} < 1$. Then the solutions $Y(n)$ to the random ordinary differential equations
		\begin{align}\label{eqn:random_ODE}
			\dif Y_t(n) = \sigma(Y_t(n))\, \dif B_t(n) ; \quad Y_0(n) = y
		\end{align}
		converge in the $\alpha$-H\"older metric for any $\frac{1}{3} < \alpha < \frac{1}{2}$ to the solution of the Stratonovich stochastic differential equation
		\begin{align*}
			\dif Y_t = \sigma(Y_t)\, \circ \dif B_t ; \quad Y_0 = y
		\end{align*}
		almost surely as $n \to \infty$.

	\end{theorem}
	
	\begin{proof}
		Let $\mathbf{B}(n)$ be the canonical lift of $B(n)$ to an $\alpha$-H\"older rough path. Then the solutions $Y(n)$ to the random ordinary differential equations \eqref{eqn:random_ODE} coincide with the solutions to the rough differential equations
		\begin{align*}
			\dif Y_t(n) = \sigma(Y_t(n))\, \dif \mathbf{B}_t(n) ; \quad Y_0(n) = y.
		\end{align*}  
		From Corollary \ref{cor:Ito_Strat_RP}, the solution $Y$ of the Stratonovich stochastic differential equation coincides almost surely with the solution to the random rough differential equation 
		\begin{align*}
			\dif Y_t = \sigma(Y_t)\, \dif \mathbf{B}^{\text{Strat}}_t;  \quad Y_0 = y.
		\end{align*}
		In the proof of Proposition \ref{geometric_brownian}, we have seen that $\varrho_{\alpha}(\mathbf{B}(n), \mathbf{B}^{\text{Strat}}) \to 0$ as $n \to \infty$. From the stability result on RDE solutions (Theorem \ref{thm:stability_RDEs}), it follows that
		\begin{align*}
			d^{\flat}_{\alpha}((Y,Y'),(Y(n),Y'(n))) \to 0
		\end{align*}
		almost surely as $n \to \infty$. In particular, $Y(n) \to Y$ almost surely as $n\to \infty$ in the $\alpha$-H\"older metric.

	\end{proof}
	
	\subsection{Rough differential equations driven by a fractional Brownian motion.}
	We will come back now to our motivating problem, i.e. the question of how to define a meaningful solution to a stochastic differential equation driven by a fractional Brownian motion $B^H$. For $H > \frac{1}{2}$, we can use Young's integration theory to solve such equations. In the case $H = \frac{1}{2}$, we can either use It\=o's theory of stochastic integration or rough paths theory as we saw in the previous section. What about $H < \frac{1}{2}$? It turns out that a similar result as seen in the proof of Proposition \ref{geometric_brownian} holds for the fractional Brownian motion, too, provided $H > \frac{1}{4}$. To formulate it, let $B^H(n)$ denote a piecewise-linear approximation of $B^H$. Since $B^H(n)$ has smooth sample paths, the canonical lift $\mathbf{B}^H(n)$ to an $\alpha$-H\"older rough path exists. With much more involved arguments as in Proposition \ref{geometric_brownian} (cf. \cite{CQ02, FV10-2}), it can be shown that $(\mathbf{B}^H(n))_{n \in \N}$ is a Cauchy sequence almost surely in the space of geometric $\alpha$-H\"older rough paths for $\alpha < H$. Since the space of geometric rough paths is complete, the sequence converges to a limit $\mathbf{B}^H$ which is then called the \emph{natural lift} of the fractional Brownian motion. This result allows to study stochastic equations driven by a fractional Brownian motion with a Hurst parameter $H > \frac{1}{4}$. There are many works in which such equations are studied, the interested reader is referred to \cite[Chapter 10]{FH20} and the comments at the end of this chapter. 
	
	A natural question is whether there is a meaningful lift in the case of $H \leq \frac{1}{4}$, too. In \cite{CQ02}, it is shown that the approach we just described here does not work for $H \leq \frac{1}{4}$ because the natural lifts $\mathbf{B}^H(n)$ will diverge in this case. To the authors’ knowledge, it is currently not clear whether a meaningful rough path lift can be defined in the regime $H \in (0,1/4]$. 

		\section{Discussion and Outlook}
	
    In these notes, we gave a brief introduction to the theory of rough paths. We emphasized its application in stochastic analysis, discussing, in particular, its ability to solve stochastic differential equations driven by a fractional Brownian motion. \smallskip
    %One of our main results was the proof of the Wong-Zakai theorem, cf. Theorem \ref{thm:wong_zakai}. \smallskip

    Rough path theory is nowadays a mature theory that found many applications in various fields of mathematics. At the end of these notes, we would like to discuss further branches of research in which rough paths theory plays a role. We are aware that the choice of topics we present here reflects our personal interests, and there are many important subjects we are not going to discuss here. In particular, we want to repeat that we will not touch the numerous applications of rough paths theory in the context of stochastic partial differential equation, a topic far beyond the scope of these notes.
    
		\begin{itemize}
            \item Gaussian rough paths and rough differential equations driven by Gaussian signals were studied extensively. The foundations were laid in the articles \cite{CQ02,FV10-2,FGGR16}, cf. also \cite[Chapter 15]{FV10} and \cite[Chapter 10]{FH20}. The continuity of the solution map, cf. Theorem \ref{thm:stability_RDEs}, allows to give an easy proof for the Freidlin-Wentzell large deviation principle and the Stroock-Varadhan support theorem \cite{LQZ02}. These theorems have natural extensions to stochastic differential equations driven by Gaussian rough paths, too \cite[Chapter 19]{FV10}.
            
            \item A famous theorem from H\"ormander characterizes second order hypoelliptic differential operators by stating a condition on the iterated Lie brackets of the involved vector fields \cite{Hor67}. This result has an equivalent formulation in terms of stochastic differential equations: H\"ormander's theorem says that if the vector fields of an SDE driven by a Brownian motion satisfy the bracket condition, the solution to the SDE obtains a smooth density at every time point $t > 0$. In \cite{Mal78}, Malliavin gave a proof of H\"ormander's theorem using a form of stochastic analysis on the Wiener space. Today, this calculus is called \emph{Malliavin calculus}. One core idea of Malliavin was to prove that the solution map to a stochastic differential equation is differentiable in certain directions of the noise. It turns out that the solution map of a rough differential equation enjoys a similar regularity \cite{CFV09}. This motivated the study of H\"ormander's theorem in the context of rough differential equations driven by Gaussian rough paths. In a series of papers, it was shown that H\"ormander's bracket condition is indeed sufficient for the solution to a rough differential equations driven by a Gaussian process to admit a smooth density \cite{CF10, CLL13, FR12b, CHLT15}. This density was further investigated in \cite{BOT14, BOZ15, Ina16,  BNOT16,  GOT20, IN21, GOT23, GOT22}

            \item In stochastic analysis, the Markov property ususally plays an important role. Different aspects for rough paths valued stochastic processes having the Markov property, known as \emph{Markovian rough paths}, were studied in \cite{FV08b, Lej06, Lej08, CO17, CO18, Che18}.

            \item In the classical texts about rough paths theory, one usually considers continuous paths exclusively. However, rough paths theory can be generalized to non-continuous paths, too, and is able to study stochastic processes with c\`adl\`ag sample paths such as L\'evy processes or general semimartingales, cf. \cite{FS13, FS17, FZ18, LP18, CF19}.

            \item Solving rough differential equations numerically can be a challenging problem. A natural numerical scheme to solve a rough differential equation can be deduced from the (formal) Taylor expansion of the solution, cf. \cite{Dav07} and \cite[Chapter 10]{FV10}. These schemes usually contain iterated integrals of, at least, order 2. Since these integrals are notoriously difficult to simulate, in particular if the driving signal is not a Brownian motion, several alternatives were studied. For instance, the \emph{simplified} or \emph{implementable Milstein scheme} replaces the iterated integral by a product of increments, cf. \cite{DNT12, FR14}. If this scheme is used in combination with Monte Carlo simulations, a complexity reduction can be obtained by using a multilevel Monte Carlo method, cf. \cite{BFRS16}. General Runge-Kutta schemes were studied in \cite{RR22}. Many articles study numerical schemes that are specifically designed to solve rough differential equations driven by a fractional Brownian motion and use some probabilistic properties of this process, cf. \cite{Nag15} for the (implicit) Crank-Nicolson scheme or \cite{LT19} for a first order Euler scheme with deterministic correction term.

            \item Expanding the solution to an ordinary differential equation leads to a so-called \emph{B-series}. The B-series expansion of a rough differential equation motivates the notion of a \emph{branched rough path} that was introduced by Gubinelli in \cite{Gub10}. A branched rough path does not only contain iterated integrals, but also integrated products of iterated integrals. The difference to a geometric rough path is that for branched rough paths, no product rule is assumed, i.e. the shuffle property in Corollary \ref{cor:shuffle_prop} does not hold for branched rough paths. For example, the product iterated integrals of the Brownian motion in It\=o-sense constitute a branched rough path, but not a geometric one. It turns out that there is a kind of embedding of the space of branched rough paths into a larger space of geometric rough paths, cf. \cite{HK15, BC19}. The geometry of branched rough paths was further studied in \cite{TZ20}. It turns out that the space of branched rough paths also form a continuous field of Banach spaces seen in Section \ref{sec:controlled_paths_fobs}, cf. \cite{GVRST22}.

            \item Studying the long-time behaviour of the solution to a rough differential equation is a natural problem. However, if the solution is non-Markovian, well established strategies fail. We would like to mention two approaches here that do not rely on the Markov property and were quite successful in this context. The first one was invented by Hairer to study ergodicity of stochastic differential equations driven by a fractional Brownian motion \cite{Hai05, HO07, HP11, HP13}. Hairer defines a structure that he calls \emph{stochastic dynamical system} (SDS) to study these equations. An SDS has certain similarities to L.~Arnold's notion of a \emph{random dynamical system} \cite{Arn98} (see below), but it is closer to the classical Markovian framework. In Hairer's theory, invariant measures can be similarly defined as for classical Markov processes. Existence and uniqueness of these measures can be proven with techniques (e.g. the coupling method) that are well-known in the Markovian world. Other researchers adopted this framework and studied, for instance, the convergence rate towards the equilibrium \cite{FP17, DPT19} or used it to study an estimator for the drift coefficient in an equation driven by a fractional Brownian motion \cite{PTV20}. Another approach is to study the random dynamical system (RDS) in the sense of L.~Arnold \cite{Arn98} that is generated by a stochastic differential equation. A rough differential equations generates an RDS whenever the driving rough paths valued process has stationary increments \cite{BRS17}. This is the case, for instance, for the fractional Brownian motion. In the theory of RDS, different objects can be defined that describe the long time behaviour of the solution to a rough differential equation. For example, one can study random attractors \cite{Duc22}, random center manifolds \cite{NK21} or random stable and unstable manifolds for rough delay equations \cite{GVRS22, GVR21}.

            \item As we already mentioned in Remark \ref{remark:signature}, the signature of a rough path is an important object that is still studied a lot. One interesting problem is to find an algorithm that reconstructs the path from a given signature effectively. This question was discussed e.g. in \cite{LX18, LX17, CDNX17, Gen17}. Due to its generalilty, the signature also plays a role in model-free mathematical finance, cf. \cite{LNPA19, LNPA20, KLA20, BHRS23,CGS23}. We already mentioned that the signature is an important object in machine learning and time series analysis, but we are unable to summarize the corresponding vast literature in these notes. Instead, we refer the reader to the overview articles \cite{CK16, LM22}.

            \item We saw in these lecture notes that the Sewing lemma (Lemma \ref{lemma:sewing}) is one of the cornerstones in rough paths theory. In the work \cite{Le20}, L\^{e} proves a stochastic version of it that he called \emph{Stochastic sewing lemma}, see also \cite[Section 4.6]{FH20}. With the Stochastic sewing lemma, it is possible to prove that certain Riemann-type sums involving random variables converge to a limit in a stochastic sense, taking into account stochastic cancellations. For instance, it is well-known that the It\=o-integral that integrates an adapted process with respect to a Brownian motion can be approximated by Riemann sums in probability, but this fact cannot be proven with the classical Sewing lemma that only looks at the regularity of the sample paths and neglects the probabilistic structure. With the Stochastic sewing lemma, however, this is possible. The Stochastic sewing lemma proved to be a very helpful tool and could be used in various settings, e.g. in the context of the regularization by noise phenomenon \cite{HP21, Hl22, Ger23} or for the analysis of numerical methods for singular SDEs \cite{BDG21, BDG23, DGL23}.

		\end{itemize}
	\subsection*{Acknowledgements}
	
	Both authors would like to thank the organizers of the XXV Brazilian School of Probability for their hospitality and generosity during our stay in Campinas.
	% \label{sec:acknowledgements}
	
	\bibliographystyle{alpha}
	\bibliography{refs}

\end{document}